\documentclass[10pt, reqno]{amsart}
\usepackage{amssymb, amsmath, amsthm}
\numberwithin{equation}{section}

\def\pa{\partial}
\newcommand{\qtq}[1]{\quad\text{#1}\quad}
\newcommand{\tz}{2^{\Z}}
\newcommand{\po}{P^a}
\newcommand{\tpo}{\tilde{P}^a}
\newcommand{\lpo}{L^p(\R^d)}

\DeclareMathOperator*{\diam}{diam}

\DeclareMathOperator*{\supp}{supp}

\newcommand{\R}{\mathbb{R}}
\newcommand{\C}{\mathbb{C}}

\newcommand{\Z}{\mathbb{Z}}

\newtheorem{theorem}{Theorem}[section]

\newtheorem{lemma}[theorem]{Lemma}

\newtheorem{proposition}[theorem]{Proposition}
\newcommand{\la}{\lambda}
\theoremstyle{definition}

\theoremstyle{remark}

\begin{document}
\title[Schr\"odinger operator with inverse-square potential]{Sobolev spaces adapted to\\the Schr\"odinger operator\\with inverse-square potential}

\author{R. Killip}
\address{Department of Mathematics, UCLA}
\email{killip@math.ucla.edu}

\author{C. Miao}
\address{Institute of Applied Physics and Computational Mathematics, Beijing 100088}
\email{miao\_changxing@iapcm.ac.cn}

\author{M. Visan}
\address{Department of Mathematics, UCLA}
\email{visan@math.ucla.edu}

\author{J. Zhang}
\address{Department of Mathematics, Beijing Institute of Technology, Beijing 100081}
\email{zhangjunyong111@sohu.com}

\author{J. Zheng}
\address{Universit\'e Nice Sophia-Antipolis, 06108 Nice Cedex 02, France}
\email{zhengjiqiang@gmail.com, zheng@unice.fr}

\begin{abstract} We study the $L^p$-theory for the Schr\"odinger operator $\mathcal L_a$ with inverse-square potential $a|x|^{-2}$.  Our main result describes when $L^p$-based Sobolev spaces defined in terms of the operator $(\mathcal L_a)^{s/2}$ agree with those defined via $(-\Delta)^{s/2}$.  We consider all regularities $0<s<2$.

In order to make the paper self-contained, we also review (with proofs) multiplier theorems, Littlewood--Paley theory, and Hardy-type inequalities associated to the operator $\mathcal L_a$.
\end{abstract}

\maketitle

 \begin{center}
 \begin{minipage}{100mm}
   { \small {\bf Key Words:}  Riesz transforms, inverse-square potential; Littlewood--Paley theory; Mikhlin Multiplier Theorem; heat kernel estimate.
      {}
   }\\
    { \small {\bf AMS Classification:}
      {35P25,  35Q55.}
      }
 \end{minipage}
 \end{center}

\section{Introduction}
In this paper, we will discuss several basic harmonic analysis questions related to the Schr\"odinger operator
\begin{equation}\label{ISP}
\mathcal{L}_a= -\Delta  + \tfrac{a}{|x|^2} \quad\text{with}\quad a\geq -\big(\tfrac{d-2}2\big)^2
\end{equation}
in dimensions $d\geq 3$.  More precisely, we interpret $\mathcal{L}_a$ as the Friedrichs extension of this operator defined initially on $C^\infty_c(\R^d\setminus\{0\})$.
We discuss this more fully in subsection~\ref{SS:La}.   The restriction $a\geq -(\tfrac{d-2}2)^2$ ensures that the operator $\mathcal{L}_a$ is positive semi-definite.

For much of what follows, it is convenient to introduce a different parameterization of the family of operators $\mathcal L_a$, namely, via
\begin{equation}\label{E:sigma defn}
\sigma:=\tfrac{d-2}2- \tfrac12\sqrt{(d-2)^2+4a} .
\end{equation}
This has the opposite sign to $a$ and ranges over $(-\infty,\frac{d-2}{2}]$ as $a$ ranges from $+\infty$ to $-(\tfrac{d-2}2)^2$.  The underlying significance of $\sigma$
will become apparent below.

The operator $\mathcal L_a$ arises frequently in mathematics and physics, commonly as a scaling limit of more complicated problems.  Several instances where
this occurs in physics are discussed in the mathematical papers \cite{BPST1,BPST2,KSWW,VZ,ZZ}; they range from combustion theory to the Dirac equation with Coulomb potential, and to the study of perturbations of classic space-time metrics such as Schwarzschild and Reissner--Nordstr\"om.

The appearance of $\mathcal L_a$  as a scaling limit (both microscopically and astronomically) is a signal of one of its unique properties: $\mathcal L_a$ is scale invariant.  In particular, the potential and Laplacian can be regarded as equally strong at every length scale.  Correspondingly, problems involving $\mathcal L_a$ are seldom amenable to simple perturbative arguments.  This is one of the reasons that we (and indeed many before us) have singled out this particular operator for in-depth study.

Our main motivation for developing harmonic analysis tools for $\mathcal L_a$, however, comes from problems in dispersive PDEs.  Specifically, the study of large-data problems at the (scaling-)critical regularity and the development of sharp scattering thresholds for focusing problems.  A key principle in attacking such problems in that one must first study the effective equations in various limiting regimes.  Given the (possibly broken, but still intrinsic) scale and space-translation invariance of such problems, one must accept that solutions may live at any possible length scale, as well as at any spatial location.  Naturally, taking a scaling limit results in a scale invariant problem.  Thus, the broad goal of demonstrating well-posedness of critical problems in general geometries rests on a thorough understanding of limiting problems, such as that associated to the inverse-square potential.  The harmonic analysis tools developed in this paper have been essential to \cite{KMVZZ:NLS}, which considers global well-posedness and scattering for both the defocusing and focusing energy-critical NLS with inverse-square potential, and to \cite{KMVZ:cubic}, in which sharp thresholds for well-posedness and scattering are established for the focusing cubic NLS with inverse-square potential.

The spectral theorem allows one to define functions of the operator $\mathcal{L}_a$ and provides a simple necessary and sufficient condition for their boundedness on $L^2$.  A sufficient condition for functions of an operator to be $L^p$-bounded is a basic prerequisite for the modern approach to many PDE problems.  In the setting of constant-coefficient differential operators, this role is played by the classical Mikhlin Multiplier Theorem.  The analogue for our operator is known:

\begin{theorem}[Mikhlin Multipliers]\label{thm:M} Suppose $m:~[0,\infty)\to\C$ satisfies
\begin{equation}\label{E:MikhlinCond}
\big|\partial^jm(\la)\big|\lesssim\la^{-j}\quad\text{for all} \quad j\geq 0
\end{equation}
and that either\\[1mm]
\hspace*{0.5em}$\bullet$\ $a\geq 0$ and $1<p<\infty$, or\\[1mm]
\hspace*{0.5em}$\bullet$\ $-(\frac{d-2}2)^2\leq a<0$ and $r_0<p<r_0':=\tfrac{d}\sigma$.\\[1mm]
Then $m(\sqrt{\mathcal L_a})$ extends to a bounded operator on $L^p(\R^d)$.
\end{theorem}

In the case $a\geq 0$, the operator $\mathcal{L}_a$ obeys Gaussian heat kernel bounds and so the result follows from general results covering this class of operators; see, for example, \cite{Alex,Hebisch}.  When $-(\frac{d-2}2)^2\leq a<0$, the heat kernel for $\mathcal{L}_a$ no longer obeys Gaussian bounds; indeed, the heat kernel is singular at the origin (see Theorem~\ref{T:heat}).  In fact, as has been observed previously, the singularity is sufficiently bad that Mikhlin-type multipliers are \emph{unbounded} outside the range of $p$ given in Theorem~\ref{thm:M}; we give the simple details in Section~\ref{S:MM}.  Thus, the stated range of $p$ is sharp, including the question of endpoints.

Several authors have developed general multiplier theorems under weaker hypotheses on the heat kernel.  In particular, the papers \cite{Blunck, COSY} explicitly verify that the multiplier theorems proved therein apply to the operators $\mathcal{L}_a$ for $a>-(\frac{d-2}2)^2$ and so verify Theorem~\ref{thm:M} in these cases.  In fact, \cite[Theorem~1.1]{Blunck} also yields Theorem~\ref{thm:M} in the endpoint case $a=-(\frac{d-2}2)^2$; to see this, one need only observe that the known heat kernel estimates (reproduced below in Theorem~\ref{T:heat}) guarantee that $\mathcal L_a$ obeys hypothesis GGE from \cite{Blunck} with exponent $p_o$ if and only if $\frac{d}{\sigma} < p_o \leq 2$.

Nevertheless, we have elected to include here a proof of Theorem~\ref{thm:M} for the more difficult regime $-(\frac{d-2}2)^2\leq a<0$.  We do this for two principal reasons: (i) This result is a key foundation for what is done in this paper. (ii) The precise formulation of Theorem~\ref{thm:M} permits us to give a simpler and more self-contained proof than those of the papers cited above.  Indeed, we are dealing with a concrete operator acting on $\R^d$.  Moreover (and more significantly), for applications to dispersive PDE, the Mikhlin condition \eqref{E:MikhlinCond} suffices.  By comparison, the papers cited above tackle the deep and difficult problem of finding minimal regularity H\"ormander-type conditions under which multiplier operators are bounded.
%For the weakest regularity assumptions applicable to $\mathcal L_a$, of which we are aware, see \cite{KunstmannUhl}.

The argument we give to verify Theorem~\ref{thm:M} is strongly influenced by \cite{Alexopoulos,CGT:finite,SikoraWright}.  The argument only requires bounds on a concrete number of derivates of the multiplier (see Theorem~\ref{thm:multthm}); however, in the interests of clarity and simplicity, we have made no effort to optimize this number.

The square function estimates of Littlewood--Paley follow as an immediate corollary of the multiplier theorem via the usual randomization argument.  Such estimates are an invaluable tool since they allow one to analyze problems one length scale at a time and then reassemble the pieces.

It is convenient for us to consider two kinds of Littlewood--Paley projections: one defined via a $C^\infty_c((0,\infty))$ multiplier and another defined as a difference of heat kernels. See Section~\ref{SS:LP} for more details.  The former notion more closely matches modern expositions of the classical translation-invariant theory, while the heat kernel version allows one to exploit heat-kernel bounds (and the semigroup property) to prove estimates.  There is no great cost associated to either choice, since the multiplier theorem permits one to readily pass back and forth between the two notions.  An example of this in action can be found in the proof of Bernstein inequalities; see Lemma~\ref{L:Bernie}.

One important application of the traditional Littlewood--Paley theory is the proof of Leibniz (=product) and chain rules for
differential operators of non-integer order.  For example, if $1<p<\infty$ and $s>0$, then
$$
\| f g \|_{H^{s,p}(\R^d)} \lesssim \| f \|_{H^{s,p_1}(\R^d)} \| g \|_{L^{p_2}(\R^d)} + \| f \|_{L^{p_3}(\R^d)}\| g \|_{H^{s,p_4}(\R^d)}
$$
whenever $\frac1p=\frac1{p_1}+\frac1{p_2}=\frac1{p_3}+\frac1{p_4}$.  For a textbook presentation of these theorems and original references, see \cite{Taylor:Tools}.

Rather than pursue a direct proof of such inequalities, we will prove a result that allows one to deduce such results directly from their Euclidean counterparts.  This is the main result of this paper:

\begin{theorem}[Equivalence of Sobolev norms]\label{thm:equivsobolev} Suppose $d\geq 3$, $a\geq -(\tfrac{d-2}2)^2$, and $0<s<2$.
If $1<p<\infty$ satisfies $\frac{s+\sigma}d<\frac1p<\min\{1, \frac{d-\sigma}d\}$, then
\begin{equation}\label{RieszT}
\big\|(-\Delta)^\frac{s}2f\big\|_{L^p}\lesssim_{d,p,s}\big\|\mathcal L_a^\frac{s}2 f\big\|_{L^p} \qtq{for all} f\in C_c^\infty(\R^d).
\end{equation}
If $\max\{\frac sd, \frac\sigma d\}<\frac1p<\min\{1, \frac{d-\sigma}d\}$, which ensures already that $1<p<\infty$, then
\begin{equation}\label{RieszT'}
\big\|\mathcal L_a^\frac{s}2f\big\|_{L^p}\lesssim_{d,p,s}\big\|(-\Delta)^\frac{s}2f\big\|_{L^p}\qtq{for all} f\in C_c^\infty(\R^d).
\end{equation}
\end{theorem}

Note that for the applications described above, both inequalities \eqref{RieszT} and \eqref{RieszT'} are equally important.  The crucial role of this result in the study of nonlinear Schr\"odinger equations with inverse-square potential is documented in \cite{KMVZZ:NLS,KMVZ:cubic}.  Here again, \eqref{RieszT} and \eqref{RieszT'} are equally important.  In the case of \cite{KMVZZ:NLS} it is also essential that Theorem~\ref{thm:equivsobolev} covers both the regimes where $s<1$ and $s>1$. 

When $s=1$, the estimate \eqref{RieszT} is known as the boundedness of Riesz transforms.  The case $s=1$ of \eqref{RieszT} was proved in
\cite{HassellLin}, excepting the endpoint case $a=-(\tfrac{d-2}2)^2$.  See also \cite{Assaad, Assaad-Ouhabaz} for prior partial results.  Moreover, the earlier paper \cite{HassellSikora} showed that for $s=1$, the range of $p$
stated above cannot be enlarged, even if one restricts $f$ to be spherically symmetric.  

When $p=2$, the equivalence at $s=1$ for $a>-(\tfrac{d-2}2)^2$ is an immediate consequence of the sharp Hardy inequality
\begin{equation}\label{SharpL2Hardy}
\bigl(\tfrac{d-2}2\bigr)^2 \int_{\R^d} \frac{|f(x)|^2}{|x|^2}\,dx \leq \int_{\R^d} |\nabla f(x)|^2\,dx,
\end{equation}
which itself is just a recapitulation of the fact that $\mathcal L_a\geq 0$ when $a\geq -(\tfrac{d-2}2)^2$.  Sharpness of the constant in \eqref{SharpL2Hardy} also shows that \eqref{RieszT} must fail when $p=2$, $s=1$, and $a=-(\tfrac{d-2}2)^2$.  This scenario coincides precisely with an excluded endpoint in Theorem~\ref{thm:equivsobolev}.

When $p=2$ and  $a > -(\tfrac{d-2}2)^2$, equivalence of Sobolev norms for $0\leq s\leq 1$ follows by complex interpolation from the endpoints $s=0$ and $s=1$.
This argument appears as  Proposition 1 in \cite{BPST1} and relies on bounds of imaginary powers of $\mathcal L_a$.  These follow from the spectral theorem in the $L^2$ setting, but are highly non-trivial for general $p$.  In \cite{SikoraWright} it is shown that Gaussian heat kernel bounds imply suitable bounds on imaginary powers.  This is used in \cite{ZZ} to prove an analogue of Theorem~\ref{thm:equivsobolev} for $a\geq 0$, albeit with a smaller set of exponents.

The proof of Theorem~\ref{thm:equivsobolev} follows a path laid down in \cite{KVZ12}:  The classical Mikhlin Multiplier Theorem together with Theorem~\ref{thm:M} allow us to reformulate the norms on either side of \eqref{RieszT} and \eqref{RieszT'} in terms of Littlewood--Paley square functions (cf. Theorem~\ref{T:sq}).  Here we use Littlewood--Paley projections formed in terms of the heat kernel, because this makes it easy to bound the difference between the kernel associated to the operator $\mathcal L_a$ and that associated to $-\Delta$.  By bounding the difference between these operators (see Lemma~\ref{lem:tedke}), we are ultimately lead to
\begin{align}\label{est:diff}
\biggl\|  \biggl(\sum_{N\in 2^{\Z}}N^{2s}\bigl|\tilde P_Nf\bigr|^2\biggr)^{\frac 12}
    - \biggl(\sum_{N\in 2^{\Z}}N^{2s}\bigl|\tilde P_N^a f\bigr|^2\biggr)^{\frac 12}\biggr\|_{L^p(\R^d)}
\lesssim \biggl\|\frac{f(x)}{|x|^s}\biggr\|_{L^p(\R^d)}.
\end{align}
This assertion is the content of Proposition~\ref{prop:lpdiffer}.  The proof of Theorem~\ref{thm:equivsobolev} is then completed by invoking the classical Hardy inequality for the Laplacian, as well as an analogue of this for the operator $\mathcal L_a$ that we prove in Proposition~\ref{P:hardy}.

This completes our description of the main topics of the paper.  The exact presentation is organized as follows.  This section contains two subsections.  In the first we review the definition of $\mathcal L_a$ and justify our choice of the Friedrichs extension.  In the second, we review some basic notation.  In Section~2,  we use the heat kernel estimate to derive bounds on the kernel of the Riesz potential. Section~3 is devoted to proving Hardy-type inequalities associated with $\mathcal L_a$.  We prove Theorem~\ref{thm:M} in Section~4. In Section~5, we develop a Littlewood--Paley theory in our setting.  We prove Theorem \ref{thm:equivsobolev} in Section~6.

\subsection{$\mathcal L_a$ as a self-adjoint operator.}\label{SS:La} Let $\mathcal L_a^\circ$ denote the natural action of $-\Delta+\smash[b]{\tfrac{a}{|x|^2}}$ on $C^\infty_c(\R^d\setminus \{0\})$.  Recall that we define $\mathcal L_a$ as the Friedrichs extension of $\mathcal L_a^\circ$.  The purpose of this subsection is to elaborate on the meaning and significance of this.  As elsewhere in the paper, we restrict attention to $d\geq 3$.

When $a\geq -(\tfrac{d-2}{2})^2$, the operator $\mathcal L_a^\circ$ is easily seen to be a positive semi-definite symmetric operator.  For example, positivity can seen via the factorization
$$
\mathcal{L}_a^\circ = \bigl(-\nabla + \sigma\tfrac{x}{|x|^2}\bigr) \cdot \bigl(\nabla + \sigma\tfrac{x}{|x|^2}\bigr)
    = -\Delta - \sigma (d-2-\sigma) \tfrac{1}{|x|^2},
$$
with $\sigma$ is as in \eqref{E:sigma defn}, which shows that for $\phi\in C^\infty_c(\R^d\setminus\{0\})$ we have
$$
\langle \phi, \mathcal{L}_a^\circ\phi\rangle = \int_{\R^d} \bigl| \nabla \phi(x) + \sigma\tfrac{x}{|x|^2} \phi(x) \bigr|^2\,dx \geq 0.
$$

The general theory of self-adjoint extensions now guarantees that there is a unique self-adjoint extension $\mathcal L_a$ of $\mathcal{L}_a^\circ$, whose form domain $Q(\mathcal L_a)=D(\sqrt{\mathcal L_a})\subseteq L^2(\R^d)$ is given by the completion of $C^\infty_c(\R^d\setminus\{0\})$ with respect to the norm
$$
||\phi||^2_{Q(\mathcal L_a)} = \int_{\R^d} \bigl| \nabla \phi \bigr|^2 + \bigl(1 + \tfrac{a}{|x|^2}\bigr) \bigl|\phi\bigr|^2\,dx
    = \int_{\R^d} \bigl| \nabla \phi + \tfrac{\sigma x}{|x|^2} \phi \bigr|^2 + \bigl|\phi\bigr|^2\,dx.
$$
This extension is known as the Friedrichs extension; it is also positive semi-definite.  No other self-adjoint extension has domain contained inside $Q(\mathcal L_a)$.  In this sense (see also below), the Friedrichs extension is singled out as having the smallest and hence least singular quadratic form domain.  For further details, see \cite[\S X.3]{RSII}.

When $a>-(\tfrac{d-2}{2})^2$, the sharp Hardy inequality \eqref{SharpL2Hardy} shows that $Q(\mathcal L_a)=H^1(\R^d)$.  When $a=-(\tfrac{d-2}{2})^2$, the Hardy inequality still guarantees $Q(\mathcal L_a)\supseteq H^1(\R^d)$; however, the reverse inclusion fails.  To see this failure, one need only consider a function $u(x)$ that is compactly supported, smooth except at the origin, and obeys
$$
u(x) = |x|^{-\frac{d-2}{2}}[\log(1/|x|)]^{-\frac{1}{2}} \qtq{for} |x| < \tfrac12.
$$
This is the first of many subtleties associated with the endpoint case.

We are able to give a seamless treatment of the endpoint case in this paper for two main reasons. First, unlike the Green's function, for example, the structure of the heat kernel is insensitive to a zero-energy resonance (in the form of an incipient eigenvalue).  The difference between the Green's function and the heat kernel is apparent even when looking at the case $a= 0$ as the dimension varies; indeed, the heat kernel varies coherently with dimension, while the Green's function has logarithmic terms in two dimensions (a signal of the zero-energy resonance in that case).  By basing our arguments around the heat kernel, we avoid these peculiarities.

The second key factor in our treatment of the endpoint case is the fact that precise estimates for the heat kernel have already been verified in this case.  This is a result of \cite{MS} that we reproduce here in Theorem~\ref{T:heat}.  For further discussion of the peculiarities and subtleties of the heat equation in the endpoint case see \cite{VZ}, as well as the references therein.

As the operator $\mathcal L_a^\circ$ is spherically symmetric, we may decompose it as a direct sum over spherical harmonics.  This is very instructive for understanding the structure of solutions and the possibility of alternate self-adjoint extensions.  To this end, we employ polar coordinates in the form $x=r\omega$ with $0\leq r<\infty$ and $\omega\in S^{d-1}$.

Consider now functions of the form $f(r) Y_\ell(\omega)$ with $f\in C^\infty_c((0,\infty))$ and $Y_\ell$ a spherical harmonic of degree $\ell\geq 0$.  The action of $\mathcal L_a^\circ$ on such functions is equivalent to the action of the Bessel operator
$$
f(r) \mapsto -\tfrac{d^2\;}{dr^2}f(r)-\tfrac{d-1}{r}\tfrac{d\;}{dr}f(r)+\tfrac{a+\ell(\ell+d-2)}{r^2}f(r)
$$
on the radial factor $f(r)$.  This is a symmetric positive semi-definite operator with respect to the $L^2(r^{d-1}\,dr)$ inner product.

An alternate form of the Bessel operator,
\begin{equation}\label{LioudBessel}
g(r)\mapsto -g''(r) + \tfrac{4\nu^2-1}{4r^2} g(r) \qtq{with} \nu^2 = \bigl(\tfrac{d-2}{2}\bigr)^2 + a + \ell(\ell+d-2),
\end{equation}
arises by employing the Liouville transformation $g(r)=r^{(d-1)/2} f(r)$.  This is a symmetric operator with respect to the $L^2(dr)$ inner product.  Note also that the condition $a\geq - (\tfrac{d-2}{2})^2$ guarantees that $\nu^2 \geq 0$.  We choose $\nu$ as the positive root.

The spectral theory of the Bessel equation has been extensively studied, typically in the form \eqref{LioudBessel}; see, for example, \cite{AkhGlaz,Titch}.  The operator is always limit point at infinity (in the sense of Weyl); this is a trivial consequence of the fact that the potential is bounded near infinity.  On the other hand, the operator is limit point at the origin if and only if $\nu\geq 1$.  To see this, we note that at every energy $z\in\C$ there is a basis of (formal) eigenfunctions whose leading asymptotics at the origin are given by
$$
g_1(r) = r^{\tfrac12}  \cdot \begin{cases} 1 &: \nu=0\\ r^{\nu} &:\nu>0\end{cases} \qtq{and}
g_2(r) = r^{\tfrac12}  \cdot \begin{cases} \log(r) &: \nu=0\\ r^{-\nu}  &:\nu>0,\end{cases}
$$
respectively.  (These functions are actually exact zero-energy eigenfunctions.)  Notice that $g_1$ is always square integrable near the origin, while this is true for $g_2$ if and only if $\nu<1$.  Thus, when $\nu\geq 1$, the operator \eqref{LioudBessel} is essentially self-adjoint and formal eigenfunctions have the asymptotics of $g_1(r)$.  When $0\leq\nu<1$, the operator \eqref{LioudBessel} admits a one-parameter family of self-adjoint extensions corresponding to possible choices of boundary condition at $r=0$.  Demanding that eigenfunctions have the asymptotics proportional to those of $g_1(r)$ corresponds to the choice of the Friedrichs extension.

Transferring the foregoing information back to our original operator $\mathcal L_a^\circ$ leads to the following conclusions:
\begin{enumerate}
\item[(a)] If $a\geq 1 - (\tfrac{d-2}{2})^2$, then $\mathcal L_a^\circ$ is essentially self-adjoint and $\mathcal L_a$ denotes the unique self-adjoint extension.
Formal eigenfunctions that are spherically symmetric have asymptotics
\begin{equation}\label{formal asymp}
u(x) = c |x|^{-\sigma} \bigl(1 + O(|x|)\bigr)
\end{equation}
as $|x|\to 0$.  Eigenfunctions at higher angular momentum decay more rapidly at the origin.

\item[(b)] If $- (\tfrac{d-2}{2})^2\leq a < 1 - (\tfrac{d-2}{2})^2$, then $\mathcal L_a^\circ$ has deficiency indices $(1,1)$ and so admits a one-parameter family of self-adjoint extensions.  Such extensions differ in their action on spherically symmetric functions, but agree on the orthogonal complement.  The Friedrichs extension is characterized by the fact that spherically symmetric (formal) eigenfunctions have asymptotics \eqref{formal asymp} at the origin.
\end{enumerate}

As mentioned earlier, the inverse-square potential is important for its appearance as a scaling limit of less singular potentials.  This also justifies our focus on the Friedrichs extension.  Indeed, tedious but elementary ODE analysis shows the following:

\begin{proposition}
Let $V\in C^\infty(\R^d)$ obey $V(x) = a|x|^{-2} + O(|x|^{-3})$ as $|x|\to\infty$ with $a\geq - (\tfrac{d-2}{2})^2$.  Then the family of Schr\"odinger operators
$
-\Delta + \lambda^2 V(\lambda x)
$
converges in strong resolvent sense to $\mathcal L_a$ as $\lambda\to\infty$.
\end{proposition}

\subsection{Notation.}
If $X, Y$ are nonnegative quantities, we write $X\lesssim Y $ or $X=O(Y)$ whenever there exists a constant $C$ such that $X\leq CY$.  We write $X \sim Y$ whenever $X\lesssim Y\lesssim X$.  This notation suffices to suppress most absolute constants, with one notable exception, namely, bounds on the rate of Gaussian decay in the heat kernel.  For this (and subsequent estimates derived from it), we use the letter $c$ to denote a positive constant, which may vary from line to line.

For any $1\leq r \leq \infty$, we write $\|\cdot\|_{r}$ for the norm in $L^{r}(\mathbb{R}^d)$ where integration is with respect to Lebesgue measure and denote by $r'$ the conjugate exponent defined via $\frac{1}{r} + \frac{1}{r'}=1$. We will also use the notations
$$
A\vee B:=\max\{A,B\} \qtq{and} A\wedge B:=\min\{A,B\}.
$$

\subsection*{Acknowledgements} We are grateful to E.~M.~Ouhabaz and an anonymous referee for references connected with Theorem~\ref{thm:M}.  R.~Killip was supported by NSF grant DMS-1265868.  He is grateful for the hospitality of the Institute of Applied Physics and Computational Mathematics, Beijing, where this project was initiated. C.~Miao was supported by NSFC grants 11171033 and 11231006.  M.~Visan was supported by NSF grant DMS-1161396.  J.~Zhang was supported by PFMEC (20121101120044), Beijing Natural Science Foundation (1144014), and NSFC grant 11401024. J.~Zheng was partly supported by the ERC grant SCAPDE.

%%%%%%%%%%%%%%%%%%%%%%%%%%%%%%%%%%%%%%%%%%%%%%%%%%%%%%%%%%%%%%%%%%%%%%%%%%%%%%%%%%%%%%%%%%%%%%%%%%%%%%%%%%%%%%%%%%%%%%%%%%%%%%%%%%%%%%%%%%%%%%

%%%%%%%%%%%%%%%%%%%%%%%%%%%%%%%%%%%                                %%%%%%%%%%%%%%%%%%%%%%%%%%%%%%%%%%%%%%%%%%%%%%%%%%%%%%%%%%%%%%%%%%%%%%

%%%%%%%%%%%%%%%%%%%%%%%%%%%%%%%%%%%%%%%%%%%%%%%%%%%%%%%%%%%%%%%%%%%%%%%%%%%%%%%%%%%%%%%%%%%%%%%%%%%%%%%%%%%%%%%%%%%%%%%%%%%%%%%%%%%%%%%%%%%%%%

\section{Heat and Riesz kernels}

We begin by recalling estimates on the heat kernel associated to the operator $\mathcal L_a$; these were found by Liskevich--Sobol \cite{LS} and Milman--Semenov \cite{MS}.

\begin{theorem}[Heat kernel bounds]\label{T:heat} Assume $d\geq 3$ and $a\geq-(\frac{d-2}2)^2$.  Then there exist positive constants $C_1,C_2$ and $c_1,c_2$ such that for all
$t>0$ and all $x,y\in \R^d\setminus\{0\}$,
\begin{equation*}
C_1\bigl( 1\vee \tfrac{\sqrt t}{|x|} \bigr)^\sigma \bigl( 1\vee \tfrac{\sqrt t}{|y|} \bigr)^\sigma t^{-\frac
d2}e^{-\frac{|x-y|^2}{c_1t}}\leq e^{-t\mathcal L_a}(x,y)\leq C_2\bigl( 1\vee \tfrac{\sqrt t}{|x|} \bigr)^\sigma \bigl( 1\vee \tfrac{\sqrt t}{|y|} \bigr)^\sigma t^{-\frac
d2}e^{-\frac{|x-y|^2}{c_2t}}.
\end{equation*}
\end{theorem}

The bounds provided by Theorem~\ref{T:heat} are an essential ingredient in the following

\begin{lemma}[Riesz kernels]\label{lem:rke} Let $d\geq 3$ and suppose $0<s<d$ and $d-s-2\sigma>0$. Then the Riesz potentials
$$
\mathcal{L}_a^{-\frac{s}2}(x,y):=\frac1{\Gamma(s/2)}\int_0^\infty e^{-t\mathcal L_a}(x,y)t^{\frac{s}2}\frac{dt}t
$$
satisfy
\begin{equation}\label{est:rke}
\mathcal{L}_a^{-\frac{s}2}(x,y)\sim |x-y|^{s-d}\Big(\tfrac{|x|}{|x-y|}\wedge\tfrac{|y|}{|x-y|}\wedge1\Big)^{-\sigma}.
\end{equation}
\end{lemma}

\begin{proof}
By symmetry, we need only consider the case when $|x|\leq|y|$.  In view of this reduction and Theorem~\ref{T:heat}, the key equivalence to be proved is
\begin{align}\label{E:cRiesz}
\int_0^\infty \tau^{\frac{s-d}2}e^{-\frac{c|x-y|^2}{\tau}} \Big(\tfrac{\sqrt{\tau}}{|x|}\vee1\Big)^\sigma\Big(\tfrac{\sqrt{\tau}}{|y|}\vee1\Big)^\sigma \tfrac{d\tau}\tau
    \sim_c |x-y|^{s-d}\Big(\tfrac{|x|}{|x-y|}\wedge1\Big)^{-\sigma}
\end{align}
for all $|x|\leq|y|$ and any $c>0$.  Making the change of variables $t=|x-y|^2/\tau$ we deduce that
\begin{equation*}
\text{LHS\eqref{E:cRiesz}} \sim |x-y|^{s-d}\big(I_1 + I_2 + I_3\bigr)
\end{equation*}
where
\begin{gather*}
I_1 := \int_{\frac{|x-y|^2}{|x|^2}}^\infty t^{\frac{d-s}2}e^{-ct}\tfrac{dt}t, \qquad I_2 := \tfrac{|x-y|^{\sigma}}{|x|^{\sigma}}\int_{\frac{|x-y|^2}{|y|^2}}^{\frac{|x-y|^2}{|x|^2}}t^{\frac{d-s-\sigma}2}e^{-ct}\tfrac{dt}t,\\
\text{and} \quad I_3:= \tfrac{|x-y|^{2\sigma}}{|x|^{\sigma}|y|^{\sigma}}\int_{0}^{\frac{|x-y|^2}{|y|^2}}t^{\frac{d-s-2\sigma}2}e^{-ct}\tfrac{dt}t.
\end{gather*}

In estimating these integrals, we will repeatedly use our hypotheses that $d-s>0$ and $d-s-2\sigma>0$.  These imply $d-s-\sigma>0$.

\noindent{\bf Case 1: $|x-y|\leq4|x|$.}  As $|x|\leq |y|$, in this case we have $|x-y|\lesssim |x|\sim |y|$.  Correspondingly,
\begin{equation*}
I_1 \sim_c 1, \qquad I_2 \lesssim  \tfrac{|x-y|^{\sigma}}{|x|^{\sigma}} \int_0^{\frac{|x-y|^2}{|x|^2}} t^{\frac{d-s-\sigma}2} \,\tfrac{dt}t \sim \tfrac{|x-y|^{d-s}}{|x|^{d-s}} \lesssim 1,
\end{equation*}
and
$$
I_3 \lesssim \tfrac{|x-y|^{2\sigma}}{|x|^{\sigma}|y|^{\sigma}}\int_{0}^{\frac{|x-y|^2}{|y|^2}}t^{\frac{d-s-2\sigma}2}\tfrac{dt}t \lesssim \tfrac{|x-y|^{d-s}}{|x|^{\sigma}|y|^{d-s-\sigma}} \lesssim 1.
$$
As $I_2$ and $I_3$ are positive, these estimates verify \eqref{E:cRiesz} in this case.

\noindent{\bf Case 2: $|x-y|\geq4|x|$.} As $|x|\leq |y|$, we have $|x-y|\sim |y| \geq |x|$ in this case.  Clearly,
\begin{align}\nonumber
I_3 \sim_c \tfrac{|x-y|^{2\sigma}}{|x|^{\sigma}|y|^{\sigma}} \sim \Big(\tfrac{|x|}{|x-y|} \Big)^{-\sigma}.
\end{align}
On the other hand,
\begin{equation*}
I_2 \lesssim \tfrac{|x-y|^{\sigma}}{|x|^{\sigma}} \int_{0}^{\infty} t^{\frac{d-s-2\sigma}2}e^{-ct}\frac{dt}t\lesssim \Big(\tfrac{|x|}{|x-y|} \Big)^{-\sigma}
\end{equation*}
and
$$
I_1 \lesssim_c \int_{\frac{|x-y|^2}{|x|^2}}^\infty e^{-ct/2}\,dt \lesssim_c  \Big(\tfrac{|x|}{|x-y|} \Big)^{-\sigma}.
$$
As $I_1$ and $I_2$ are positive, these estimates verify \eqref{E:cRiesz} in this case.
\end{proof}

%%%%%%%%%%%%%%%%%%%%%%%%%%%%%%%%%%%%%%%%%%%%%%%%%%%%%%%%%%%%%%%%%%%%%%%%%%%%%%%%%%%%%%%%%%%%%%%%%%%%%%%%%%%%%%%%%%%%%%%%%%%%%%%%%%%%%%%%%%%%%%

%%%%%%%%%%%%%%%%%%%%%%%%%%%%%%%%%%%                                %%%%%%%%%%%%%%%%%%%%%%%%%%%%%%%%%%%%%%%%%%%%%%%%%%%%%%%%%%%%%%%%%%%%%%

%%%%%%%%%%%%%%%%%%%%%%%%%%%%%%%%%%%%%%%%%%%%%%%%%%%%%%%%%%%%%%%%%%%%%%%%%%%%%%%%%%%%%%%%%%%%%%%%%%%%%%%%%%%%%%%%%%%%%%%%%%%%%%%%%%%%%%%%%%%%%%

\section{Hardy inequality}\label{S:Hardy}

In this section, we establish Hardy-type inequalities adapted to the operator $\mathcal L_a$.  In the absence of a potential (i.e., for $a=0$), these are well-known (cf. \cite[\S II.16]{Triebel:SoF}):

\begin{lemma}[Hardy inequality for $-\Delta$]\label{lem:classhardy}  For $1<p<\infty$ and $0\leq s<{\frac d p}$,
\begin{align}\label{hardy}
\big\||x|^{-s}f(x)\big\|_{L^p(\R^d)}\lesssim \big\|(-\Delta)^\frac{s}2f\big\|_{L^p(\R^d)}.
\end{align}
\end{lemma}

The range of $s$ in Lemma~\ref{lem:classhardy} is sharp; indeed, failure for $s=d/p$ can be seen when $f$ is any Schwartz function that does not vanish at the origin.
The main result of this section is the analogue of Lemma~\ref{lem:classhardy} in the presence of an inverse-square potential:

\begin{proposition}[Hardy inequality for $\mathcal L_a$]\label{P:hardy} Suppose $d\geq 3$, $0<s<d$, $d-s-2\sigma>0$, and $1<p<\infty$.  Then
\begin{equation}\label{est:hardy}
\big\||x|^{-s}f(x)\big\|_{L^p(\R^d)}\lesssim \big\|\mathcal L_a^\frac{s}2f\big\|_{L^p(\R^d)}
\end{equation}
holds if and only if
\begin{equation}\label{est:hardy_hyp}
s+\sigma < \tfrac dp < d-\sigma.
\end{equation}
\end{proposition}

When $a>0$ (or equivalently $\sigma<0$), we see that the Hardy inequality holds for a larger range of pairs $(s,p)$ than in the case $a=0$; when $-(\frac{d-2}2)^2\leq a<0$, the range is strictly smaller.  Note that under the over-arching hypothesis $d-s-2\sigma>0$, we are guaranteed that $s+\sigma < d-\sigma$.  Nonetheless, it is possible that $s+\sigma<0$ or that
$d-\sigma>d$, which trivialize the corresponding part of \eqref{est:hardy_hyp}.

The proof of Proposition~\ref{P:hardy} is built on the following variant of Schur's test; see \cite{KVZ12} for historical references and the (trivial) proof.

\begin{lemma}[Schur's test with weights]\label{lem:schur} Suppose $(X,d\mu)$ and $(Y,d\nu)$ are measure spaces and let $w(x,y)$ be a positive
measurable function defined on $X\times Y$. Let $K(x,y):~X\times Y\to\C$ satisfy
\begin{align}\label{est:scha1}
\sup_{x\in X}\int_Y
w(x,y)^\frac1p\big|K(x,y)\big|d\nu(y)=C_0<\infty,\\\label{est:scha2}
\sup_{y\in Y}\int_X
w(x,y)^{-\frac1{p'}}\big|K(x,y)\big|d\mu(x)=C_1<\infty,
\end{align}
for some $1<p<\infty$. Then the operator $T$ defined by
$$Tf(x)=\int_Y K(x,y)f(y)d\nu(y)$$
is a bounded operator from $L^p(Y,d\nu)$ to  $L^p(X,d\mu)$. In particular,
\begin{equation}
\big\|Tf\big\|_{ L^p(X,d\mu)}\leq
C_0^\frac1{p'}C_1^\frac1p\|f\|_{L^p(Y,d\nu)}.
\end{equation}
\end{lemma}

\begin{proof}[Proof of Proposition \ref{P:hardy}.]
The estimate \eqref{est:hardy} is equivalent to
\begin{equation}\label{est:hardyequi}
\Big\||x|^{-s}\mathcal L_a^{-\frac{s}2}g\Big\|_{L^p(\R^d)}\lesssim\|g\|_{L^p(\R^d)}.
\end{equation}

By Lemma~\ref{lem:rke}, to prove \eqref{est:hardyequi} it suffices to show that the kernel
\begin{equation}
K_{\sigma}(x,y):=|x|^{-s}|x-y|^{s-d}\Big(\frac{|x|}{|x-y|}\wedge\frac{|y|}{|x-y|}\wedge1\Big)^{-\sigma}
\end{equation}
defines a bounded operator on $L^p(\R^d)$. To this end, we divide the kernel into three pieces.

\noindent{\bf Case 1: $|x-y|\leq 4(|x|\wedge|y|).$} In this case, we have $|x|\sim|y|$ and the kernel becomes
$$
K_\sigma(x,y)=|x|^{-s}|x-y|^{s-d}.
$$
It is easily seen that
$$
\int_{|x-y|\leq 4|x|} K_\sigma(x,y)dy+\int_{|x-y|\leq 4|y|} K_{\sigma}(x,y)dx\lesssim1,
$$
and so $L^p$-boundness on this region follows immediately from Lemma~\ref{lem:schur} with $w(x,y)\equiv1$.

\noindent{\bf Case 2: $4|x|\leq|x-y|\leq4|y|$.} In this region, we have $|x-y|\sim|y|\geq|x|$ and the kernel becomes
$$
K_\sigma(x,y)\sim|x|^{-s-\sigma}|x-y|^{\sigma+s-d}\sim |x|^{-s-\sigma}|y|^{\sigma+s-d}.
$$
Let the weight $w(x,y)$ be defined by
$$
w(x,y)=\Big(\frac{|x|}{|y|}\Big)^\alpha \quad\text{with}\quad p(s+\sigma)<\alpha<p'(d-s-\sigma).
$$
The assumption $p(s+\sigma)<d$ guarantees that it is possible to chose such an $\alpha$.

As
\begin{equation*}
\int_{4|x|\leq|x-y|\leq4|y|}w(x,y)^\frac1p\big|K_\sigma(x,y)\big|dy \lesssim|x|^{-s-\sigma+\frac{\alpha}p}\int_{|y|\geq|x|}|y|^{\sigma+s-d-\frac{\alpha}p} dy\lesssim1,
\end{equation*}
the hypothesis \eqref{est:scha1} is satisfied in this region.  The fact that hypothesis \eqref{est:scha2} is also satisfied follows from
\begin{align*}
\int_{4|x|\leq|x-y|\leq4|y|}w(x,y)^{-\frac1{p'}}\big|K_\sigma(x,y)\big|dx
&\lesssim|y|^{\sigma+s-d+\frac{\alpha}{p'}}\int_{|x|\leq|y|}|x|^{-s-\sigma-\frac{\alpha}{p'}}dx\lesssim1.
\end{align*}
Thus, an application of Lemma \ref{lem:schur} yields $L^p$-boundedness on this region.

\noindent {\bf Case 3: $4|y|\leq|x-y|\leq4|x|$.} On this region, we have $|x-y|\sim|x|\geq|y|$ and the kernel becomes
$$
K_\sigma (x,y)\sim |x|^{-s}|y|^{-\sigma}|x-y|^{\sigma+s-d}\sim |x|^{\sigma-d}|y|^{-\sigma}.
$$

\noindent{\bf Subcase 3(a): $\sigma<0$.} We note that
$$
\int_{4|y|\leq|x-y|\leq4|x|}\left|K_\sigma(x,y)\right|dy\lesssim|x|^{\sigma-d}\int_{|y|\leq|x|}|y|^{-\sigma}dy\lesssim1
$$
and
$$
\int_{4|y|\leq|x-y|\leq4|x|}\left|K_\sigma(x,y)\right|dx\lesssim|y|^{-\sigma}\int_{|x|\geq|y|}|x|^{\sigma-d}dx\lesssim1.
$$
$L^p$-boundness in this setting follows immediately from Lemma \ref{lem:schur}.

\noindent{\bf Subcase 3(b): $\sigma>0$.}  In this case, we use Lemma \ref{lem:schur} with weight given by
$$
w(x,y)=\Big(\frac{|x|}{|y|}\Big)^\alpha \quad\text{with}\quad p'\sigma<\alpha<p(d-\sigma).
$$
The assumption $p>\tfrac{d}{d-\sigma}$ guarantees that it is possible to chose such an $\alpha$.

The hypothesis \eqref{est:scha1} follows from
$$
\int_{4|y|\leq|x-y|\leq4|x|}w(x,y)^\frac1p\left|K_\sigma(x,y)\right|dy\lesssim |x|^{\sigma-d+\frac{\alpha}p}\int_{|y|\leq|x|}|y|^{-\sigma-\frac{\alpha}p}dy\lesssim1,
$$
while  hypothesis \eqref{est:scha2} is deduced from
\begin{align*}
\int_{4|y|\leq|x-y|\leq4|x|}w(x,y)^{-\frac1{p'}}\left|K_\sigma(x,y)\right|dx\lesssim|y|^{-\sigma+\frac{\alpha}{p'}}\int_{|x|\geq |y|}|x|^{\sigma-d-\frac{\alpha}{p'}}dx\lesssim1.
\end{align*}

This completes the proof of \eqref{est:hardy}.  Next, we will prove the sharpness of this inequality by constructing counterexamples.

\noindent {\bf Failure of \eqref{est:hardy} for $(s+\sigma)p\geq d$.} Let $|x_0|\gg1$ and let $\varphi\in C^\infty_c(\R^d)$ be non-negative, supported in $B(x_0,1)$, satisfying $\varphi\equiv1$ on $B(x_0,1/2)$.  Let $f:=\mathcal L_a^{-\frac s2} \varphi$.  Then
$$
\big\|\mathcal L_a^\frac{s}2f\big\|_{L^p}=\|\varphi\|_{L^p}\lesssim1.
$$
On the other hand, using Lemma~\ref{lem:rke} we see that for $|x|\leq 1$,
$$
f(x)=[\mathcal L_a^{-\frac{s}2}\varphi] (x)=\int \!\mathcal{L}_a^{-\frac{s}2}(x,y)\varphi(y)dy
    \gtrsim  \!\!\int_{|y-x_0|\leq \frac12} \frac{|x-y|^{s+\sigma-d}}{|x|^{\sigma}}\,dy \gtrsim_{x_0}|x|^{-\sigma}.
$$
Hence,
$$
\big\||x|^{-s}f(x)\big\|_{L^p(\R^d)} \gtrsim_{|x_0|} \big\||x|^{-(s+\sigma)}\big\|_{L^p(|x|\leq1)}=\infty \quad\text{whenever}\quad (s+\sigma)p\geq d.
$$

\noindent {\bf Failure of \eqref{est:hardy} for $p\leq \tfrac{d}{d-\sigma}$.} Let $\varphi$ be a non-negative bump function supported in $B(0,1)$ satisfying $\varphi\equiv1$ on $B(0,1/2)$.  Let $f:=\mathcal L_a^{-\frac s2} \varphi$. Then
$$
\big\|\mathcal L_a^\frac{s}2f\big\|_{L^p}=\|\varphi\|_{L^p}\lesssim1.
$$
On the other hand, using Lemma~\ref{lem:rke} we see that for $|x|\geq10$,
$$
f(x)=[\mathcal L_a^{-\frac{s}2}\varphi](x)=\int\mathcal{L}_a^{-\frac{s}2}(x,y)\varphi(y)dy
    \gtrsim |x|^{s+\sigma-d}\int_{|y|\leq\frac12} |y|^{-\sigma}dy\gtrsim|x|^{s+\sigma-d}.
$$
Hence,
$$
\big\||x|^{-s}f(x)\big\|_{L^p(\R^d)}\gtrsim\big\||x|^{\sigma-d}\big\|_{L^p(|x|\geq10)}=\infty \quad\text{whenever}\quad p\leq \tfrac{d}{d-\sigma}.
$$

This completes the proof of the lemma.
\end{proof}

%%%%%%%%%%%%%%%%%%%%%%%%%%%%%%%%%%%%%%%%%%%%%%%%%%%%%%%%%%%%%%%%%%%%%%%%%%%%%%%%%%%%%%%%%%%%%%%%%%%%%%%%%%%%%%%%%%%%%%%%%%%%%%%%%%%%%%%%%%%%%%

%%%%%%%%%%%%%%%%%%%%%%%%%%%%%%%%%%%                                %%%%%%%%%%%%%%%%%%%%%%%%%%%%%%%%%%%%%%%%%%%%%%%%%%%%%%%%%%%%%%%%%%%%%%

%%%%%%%%%%%%%%%%%%%%%%%%%%%%%%%%%%%%%%%%%%%%%%%%%%%%%%%%%%%%%%%%%%%%%%%%%%%%%%%%%%%%%%%%%%%%%%%%%%%%%%%%%%%%%%%%%%%%%%%%%%%%%%%%%%%%%%%%%%%%%%

\section{Multiplier theorem}\label{S:MM}

As discussed in the Introduction, the purpose of this section is to give a self-contained proof of Theorem~\ref{thm:M} in the case $-(\frac{d-2}2)^2\leq a<0$:

\begin{theorem}[Mikhlin Multipliers]\label{thm:multthm} Fix $-(\frac{d-2}2)^2\leq a<0$ and suppose that $m:~[0,\infty)\to\C$ satisfies
\begin{equation}\label{est:multthm}
\big|\partial^jm(\la)\big|\lesssim\la^{-j}\quad\text{for all} \quad 0\leq j\leq3\big\lfloor\tfrac{d}4\big\rfloor+3.
\end{equation}
Then $m(\sqrt{\mathcal L_a})$, which we define via the $L^2$ functional calculus, extends uniquely from $L^p(\R^d)\cap L^2(\R^d)$ to a bounded operator on $L^p(\R^d)$, for all $r_0<p<r_0':=\tfrac{d}\sigma$.
\end{theorem}

At the end of this section we will see that the restriction $r_0<p<r_0'$ is sharp.  Indeed, the operator $e^{-\mathcal L_a}$ is unbounded on $L^p$ for the complementary ranges of $p$ (including endpoints).  As noted in the Introduction, the number of derivatives required in \eqref{est:multthm} is far from sharp; we are merely making explicit what appears naturally in the proof below.

\begin{proof}[Proof of Theorem~\ref{thm:multthm}]
By the spectral theorem, the operator $T:=m(\sqrt{\mathcal L_a})$ is bounded on $L^2$.  Thus, using the Marcinkiewicz interpolation theorem and a duality argument, it suffices to show that $T$ is of weak-type $(q,q)$ whenever $r_0<q<2$, that is,
\begin{equation}\label{est:setbouned}
\big|\{x:~|Tf(x)|>h\}\big|\lesssim h^{-q}\|f\|_{L^q(\R^d)}^q \quad\text{for all}\quad h>0.
\end{equation}
Our argument for proving \eqref{est:setbouned} is inspired by that used in \cite[Theorem 3.1]{KVZ12}, which in turn is a synthesis of older techniques.  These previous incarnations were simpler due to the presence of Gaussian heat kernel bounds.

Applying the Calder\'on--Zygmund decomposition to $|f|^q$ at height $h^q$ yields a family of dyadic cubes $\{Q_k\}_k$, which allow us to decompose the original function $f$ as $f=g+b$ where $b=\sum_k b_k$ and $b_k=\chi_{Q_k}f$.  By construction,
\begin{equation}\label{est:czdecomp}
|g|\leq h\text{ a.e. } \quad \text{and}\quad |Q_k|\leq\frac1{h^q}\int_{Q_k}|f(x)|^qdx\leq 2^d|Q_k|.
\end{equation}
Note that by H\"older's inequality and \eqref{est:czdecomp},
\begin{equation}\label{est:l1est}
\int_{Q_k}|f(x)|dx\lesssim \|f\|_{L^q(Q_k)} |Q_k|^{\frac1{q'}}\lesssim h|Q_k|\lesssim h^{1-q}\int_{Q_k}|f(x)|^qdx.
\end{equation}

In the decomposition above, the `bad' parts $b_k$ do not obey any cancellation condition.  To remedy this, we further decompose $b_k=g_k+\tilde b_k$ according to the following definition:
\begin{equation}\label{est:badpart}
\tilde b_k:=(1-e^{-r_k^2\mathcal L_a})^\mu b_k \quad \text{and}\quad g_k:=\bigl[1-(1-e^{-r_k^2\mathcal L_a})^\mu\bigr] b_k= \sum_{\nu=1}^\mu c_\nu e^{-\nu r_k^2\mathcal L_a} b_k,
\end{equation}
where $r_k$ denotes the radius of $Q_k$ and $\mu:=\lfloor \frac d4\rfloor +1$.  As $f=g+\sum_k g_k+\sum_k \tilde b_k$, so $Tf=Tg+\sum_kTg_k+\sum_kT\tilde b_k$, and then
$$
\big\{|Tf|>h\big\}\subseteq \left\{|Tg|>\tfrac13h\right\}\cup\big\{\big|T\sum_kg_k\big|>\tfrac13h\big\}\cup\big\{\big|T\sum_k\tilde b_k\big|>\tfrac13h\big\}.
$$

By the Chebyshev inequality, the boundedness of $T$ in $L^2$, and \eqref{est:czdecomp},
\begin{align*}
\big|\big\{|Tg|>\tfrac13h\big\}\big|\lesssim h^{-2}\|Tg\|_{L^2}^2\lesssim h^{-2}\|g\|_{L^2}^2\lesssim h^{-q} \|g\|_{L^q}^q\lesssim h^{-q} \|f\|_{L^q}^q.
\end{align*}
Thus, the contribution of this term to \eqref{est:setbouned} is acceptable.

Arguing similarly, we estimate
\begin{align}\label{averages}
\big|\big\{\big|T\sum_kg_k\big|>\tfrac13h\big\}\big|\lesssim h^{-2} \big\|T\sum_kg_k\big\|_{L^2}^2\lesssim h^{-2}\big\|\sum_kg_k\big\|_{L^2}^2.
\end{align}

Using the heat kernel estimate given by Theorem~\ref{T:heat}, we obtain
\begin{align}\label{averages'}
\big\|\sum_kg_k\big\|_{L^2}^2&=\sum_{\nu, \nu'}c_\nu c_{\nu'}\sum_{k,l}\big\langle b_k,e^{-(\nu r_k^2+\nu' r_l^2)\mathcal L_a}b_l\big\rangle\\
&\lesssim\sum_{r_k\geq r_l}r_k^{-d}\int_{Q_l}\int_{Q_k} \big(\tfrac{r_k}{|x|}\vee1\big)^\sigma |b_k(x)|e^{-\frac{|x-y|^2}{cr_k^2}}
    \big(\tfrac{r_k}{|y|}\vee1\big)^\sigma |b_l(y)|\,dx\,dy. \notag
\end{align}
To proceed from here, we first freeze $k$ and $x\in Q_k$ and focus on
\begin{align}
\sum_{l:r_l\leq r_k} \int_{Q_l} e^{-\frac{|x-y|^2}{cr_k^2}} \big(\tfrac{r_k}{|y|}\vee1\big)^\sigma |b_l(y)|\,dy
&\lesssim \sum_{l:r_l\leq r_k} \int_{Q_l} e^{-\frac{|x-y|^2}{cr_k^2}} |b_l(y)|\,dy \label{ly1}\\
&\quad + \sum_{l:Q_l\subset B(0,2r_k)} \int_{Q_l} \big(\tfrac{r_k}{|y|}\big)^\sigma |b_l(y)|\,dy. \label{ly2}
\end{align}
Here we used that $Q_l \cap B(0,r_k) \neq \varnothing$ implies $Q_l\subseteq B(0,2r_k)$ because $r_l\leq r_k$.

To estimate RHS\eqref{ly1} we first observe that since $\diam(Q_l) \leq 2 r_k$,
$$
|x-y|^2 \geq \tfrac12|x-y'|^2-4r_k^2 \qtq{for all} y,y'\in Q_l.
$$
Using this and then \eqref{est:czdecomp} and \eqref{est:l1est}, we see that
\begin{align*}
\text{RHS\eqref{ly1}} &\lesssim \sum_{l:r_l\leq r_k} \|b_l\|_{L^1} \frac{1}{|Q_l|} \int_{Q_l} e^{-\frac{|x-y'|^2}{2cr_k^2}}\,dy' \\
&\lesssim h \sum_{l:r_l\leq r_k} \int_{Q_l} e^{-\frac{|x-y'|^2}{2cr_k^2}}\,dy' \\
&\lesssim h r_k^d .
\end{align*}
On the other hand, applying H\"older (in $l$ and $y$) and then \eqref{est:czdecomp} yields
\begin{align*}
\text{RHS\eqref{ly2}} &\lesssim \Biggl[\; \sum_{l:Q_l\subset B(0,2r_k)} \int_{Q_l} \big(\tfrac{r_k}{|y|}\big)^{\sigma q'} \,dy \Biggr]^{1/q'}
    \Biggl[\; \sum_{l:Q_l\subset B(0,2r_k)} \int_{Q_l} |b_l(y)|^q \,dy \Biggr]^{1/q} \\
&\lesssim \Biggl[\; \int_{B(0,2r_k)} \big(\tfrac{r_k}{|y|}\big)^{\sigma q'} \,dy \Biggr]^{1/q'}
    \Biggl[\; \sum_{l:Q_l\subset B(0,2r_k)} h^q |Q_l| \Biggr]^{1/q} \\
&\lesssim h r_k^d .
\end{align*}
Note that it was important here that $q<r_0$ since this guarantees $\sigma q' <d$.

Plugging this new information back into \eqref{averages'} and then using H\"older's inequality and \eqref{est:czdecomp} we deduce that
\begin{align*}
\big\|\sum_kg_k\big\|_{L^2}^2 &\lesssim  h \sum_k \int_{Q_k} \big(\tfrac{r_k}{|x|}\vee1\big)^\sigma |b_k(x)| \,dx \\
&\lesssim  h \Biggl[ \sum_{k} \int_{Q_k}
\big(\tfrac{r_k}{|x|}\vee1\big)^{\sigma q'} \,dx \Biggr]^{1/q'}
    \Biggl[\; \sum_{k} \int_{Q_k} |b_k(x)|^q \,dx \Biggr]^{1/q} \\
&\lesssim  h \Biggl[ \sum_{k} |Q_k| \Biggr]^{1/q'} \| f \|_{L^q} \\
&\lesssim h^{2-q} \| f \|_{L^q}^q
\end{align*}
In view of \eqref{averages}, this shows that the contribution from the functions $g_k$ is acceptable.

It remains to estimate the contribution of $\big\{\big|T\sum_k\tilde b_k\big|>\tfrac13h\big\}$.  Let $Q_k^*$ be the $2\sqrt{d}$ dilate of $Q_k$.   As
$$
\big\{\big|T\sum_k\tilde b_k\big|>\tfrac13h\big\}\subset \cup_jQ_j^\ast\cup\big\{x\in\R^d\backslash \cup_jQ_j^\ast:~\big|T\sum_k\tilde b_k\big|>\tfrac13h\big\},
$$
by Chebyshev's inequality and \eqref{est:czdecomp},
\begin{align*}
\big|\big\{\big|T\sum_k\tilde b_k\big|>\tfrac13h\big\}\big|
&\lesssim\sum_j|Q_j^\ast|+h^{-1}\sum_k\big\|T\tilde b_k\big\|_{L^1(\R^d\backslash Q_k^\ast)}\\
&\lesssim h^{-q}\|f\|_{L^q}^q+h^{-1}\sum_k\big\|T\tilde b_k\big\|_{L^1(\R^d\backslash Q_k^\ast)}.
\end{align*}
Therefore, to complete the proof of the theorem we need only show that
\begin{equation}\label{est:mhlreduce}
\big\|T\tilde b_k\big\|_{L^1(\R^d\backslash Q_k^\ast)}\lesssim
h^{1-q} \|b_k\|_{L^q}^q.
\end{equation}
To this end, we divide the region $\R^d\backslash Q_k^\ast$ into dyadic annuli of the form  $R< {\rm dist}\{x,Q_k\}\leq 2R$ for $r_k\leq R\in 2^\Z$.  We will prove the following $L^2$ estimate:
\begin{equation}\label{equ:guije}
\big\|T\tilde b_k\big\|_{L^2({\rm dist}\{x,Q_k\}> R)}\lesssim \big(\tfrac{r_k}R\big)^{2\mu}R^{-d(\frac12-\frac1{q'})}\|b_k\|_{L^q},
\end{equation}
with the implicit constant independent of $k$ and $R\geq r_k>0$.  Claim \eqref{est:mhlreduce} follows immediately from this, H\"older, \eqref{est:czdecomp}, and \eqref{est:l1est}:
\begin{align*}
\big\|T\tilde b_k\big\|_{L^1(\R^d\backslash Q_k^\ast)}
&=\sum_{R\geq r_k}\big\|T\tilde b_k\big\|_{L^1(R< {\rm dist}\{x,Q_k\}\leq 2R)}\\
&\lesssim\sum_{R\geq r_k}R^\frac{d}2\big\|T\tilde b_k\big\|_{L^2({\rm dist}\{x,Q_k\}>R)}\\
&\lesssim\sum_{R\geq r_k}R^\frac{d}2\big(\tfrac{r_k}R\big)^{2\mu} R^{-d(\frac12-\frac1{q'})} \|b_k\|_{L^q}\\
&\lesssim r_k^\frac{d}{q'}\|b_k\|_{L^q} \lesssim h^{1-q}
\|b_k\|_{L^q}^q.
\end{align*}
To guarantee that the sum above converges, we need $\frac{d}{q'}<2\mu$, which is satisfied under our hypotheses.

We now turn to \eqref{equ:guije} and write
\begin{equation*}
(T\tilde b_k)(x)= \int_{Q_k} \bigl[m(\sqrt{\mathcal L_a})(1-e^{-r_k^2\mathcal L_a})^\mu\bigr](x,y)b_k(y)dy.
\end{equation*}

Let $a(\la):=m(\la)(1-e^{-r_k^2\la^2})^\mu$, which we extend to all of $\R$ as an even function.  Using \eqref{est:multthm}, it is easy to check that
\begin{equation}\label{est:apartes}
\big|\pa^ja(\la)\big|\lesssim |\la|^{-j}\left(1\wedge r_k|\la|\right)^{2\mu}\quad \text{for all}\quad 0\leq j\leq3\lfloor\tfrac{d}4\rfloor+3.
\end{equation}

Now let $\varphi$ be a smooth, positive, even function, supported on $[-\frac12,\frac12]$, and such that $\varphi(\tau)=1$ for $|\tau|<\frac14$. Define
$\check\varphi$ and $\check\varphi_R$ via
\begin{equation*}
{\check\varphi}_R(\lambda):= R {\check\varphi}(R\lambda) := \frac1 {2\pi}\int e^{i\lambda\tau} \varphi\bigl(\tfrac\tau R\bigr)\,d\tau.
\end{equation*}
%Taking the Fourier transform  yields
%\begin{equation*}
%\varphi\bigl(\tfrac\tau R\bigr) =  \int e^{-i\lambda\tau}{\check\varphi}_R(\lambda)\,d\lambda.
%\end{equation*}
Since $a$ and $\varphi$ are even,
$$
a_1(\lambda) := (a * \check\varphi_R) (\lambda) = \pi^{-1}\int_0^\infty \cos(\lambda\tau) \hat{a}(\tau) \varphi\bigl(\tfrac\tau R\bigr)\,d\tau.
$$
The wave equation with inverse-square potential $u_{tt}+\mathcal L_a u=0$ obeys finite speed of propagation; see, for example, \cite{Sikora} or \cite[\S3]{CS}.
Noting that $\varphi(\tau/ R)$ is supported on the set $\{\tau:~|\tau|\leq\frac{R}2\}$, we therefore obtain
$$
\supp\Bigl( a_1\bigl(\sqrt{\mathcal L_a}\bigr)\delta_{y} \Bigr)\subseteq \bigcup_{\tau\leq\frac R2} \supp\Bigl(\cos\bigl(\tau\sqrt{\mathcal L_a}\bigr)\delta_{y} \Bigr) \subseteq B(y,\tfrac12 R).
$$
Thus, this part of the multiplier $a$ does not contribute to LHS\eqref{equ:guije}.

We now consider the remaining part of the multiplier $a$, namely,
$$
a_2(\lambda) := a_1(\lambda)-a(\lambda) = \int [a(\theta)-a(\lambda)] \check\varphi_R(\lambda-\theta)\,d\theta.
$$
When $|\lambda| \leq R^{-1}$,  using \eqref{est:apartes} and the rapid decay of $\check\varphi$, we get
\begin{align*}
\Big|\int a(\lambda)\check\varphi_R(\lambda-\theta)\,d\theta\Big|\lesssim&(1\wedge r_k|\lambda|)^{2\mu}\lesssim(1\wedge r_k|\lambda|)^{2\mu}(|\lambda| R)^{-2\mu},
\end{align*}
and
\begin{align*}
\Big|\int a(\theta)\check\varphi_R(\lambda-\theta)\,d\theta\Big| &\lesssim(1\wedge r_k|\lambda|)^{2\mu}(|\lambda| R)^{-2\mu}.
\end{align*}
Thus
\begin{equation}\label{a_2 bound 1}
|a_2(\lambda)| \lesssim (1\wedge r_k|\lambda|)^{2\mu}(|\lambda| R)^{-2\mu} \quad\text{when} \quad |\lambda| \leq R^{-1}.
\end{equation}

When $|\lambda| \geq R^{-1}$, expanding $a(\theta)$ in a Taylor series to order $j-1=3\lfloor\frac d4\rfloor+2$, we write
$$
a(\theta)-a(\lambda)=P_j(\theta) + {\mathcal E}(\theta) \qtq{where} P_j(\theta) := \sum_{\ell=1}^{j-1} \frac{a^{(\ell)}(\lambda)}{\ell!} (\theta-\lambda)^\ell
$$
and ${\mathcal E}$ denotes the error, which we estimate using \eqref{est:apartes} as follows:
\begin{equation*}
|{\mathcal E}(\theta)| \leq |a(\theta)| + |a(\lambda)| + |P_j(\theta)| \lesssim (1\wedge r_k|\lambda|)^{2\mu} \bigl|\tfrac{\theta-\lambda}{\lambda}\bigr|^j \quad\text{when}\quad |\theta-\lambda| > \tfrac12|\lambda|
\end{equation*}
(note $2\mu\leq 3\mu=j$) and
\begin{equation*}
|{\mathcal E}(\theta)| \leq \bigl\| a^{(j)} \bigr\|_{L^\infty([\frac\lambda2,\frac{3\lambda}2])} |\theta - \lambda|^j \lesssim (1\wedge r_k|\lambda|)^{2\mu}\bigl|\tfrac{\theta-\lambda}{\lambda}\bigr|^j
    \quad\text{when}\quad |\theta-\lambda| \leq \tfrac12|\lambda|.
\end{equation*}

For any $\ell\geq1$,
$$
\int(\theta-\lambda)^\ell\check\varphi_R(\theta-\lambda)\,d\theta=(\partial^\ell\varphi_R)^\vee (0)= 0.
$$
Thus $P_j(\theta)$ makes no contribution to the convolution defining $a_2(\lambda)$ and so
\begin{equation}\label{a_2 bound 2}
|a_2(\lambda)| \lesssim  \int |{\mathcal E}(\theta)||\check\varphi_R(\lambda-\theta)|\,d\theta \lesssim (1\wedge r_k|\lambda|)^{2\mu} (|\lambda|R)^{-j}
    \quad\text{when} \quad |\lambda| \geq R^{-1}.
\end{equation}

Combining \eqref{a_2 bound 1} with \eqref{a_2 bound 2} and the fact that $R\geq r_k$, we deduce that
\begin{align*}
|a_2(\lambda)|
&\lesssim\Bigl(\tfrac{1 \wedge r_k|\lambda|}{|\lambda| R}\Bigr)^{2\mu} (1+R^2\lambda^2)^{\frac{2\mu-j}2}
\lesssim  \bigl(\tfrac{r_k}R)^{2\mu} \int_0^\infty \bigl(\tfrac t{R^2}\bigr)^{\frac{j-2\mu}2} e^{-t/R^{2}} e^{-t\lambda^2}\tfrac{dt}{t}
\end{align*}
and so, by the spectral theorem followed by the triangle inequality,
\begin{align}\label{est:a2est}
\big\|a_2(\sqrt{\mathcal L_a})b_k\big\|_{L^2(\R^d)}
    &\lesssim \bigl(\tfrac{r_k}R)^{2\mu}\int_0^\infty \bigl(\tfrac t{R^2}\bigr)^{\frac{j-2\mu}2} e^{-t/R^{2}} \big\|e^{-t\mathcal L_a}b_k\big\|_{L^2}\tfrac{dt}{t}.
\end{align}
To proceed from here we need the following estimate, which we will justify later:
\begin{align}\label{I}
\big\|e^{-t\mathcal L_a}b_k\big\|_{L^2} \lesssim t^{-\frac{d}4} \bigl( t + r_k^2 \bigr)^\frac{d}{2q'} \|b_k\|_{L^q}.
\end{align}
Substituting this estimate into \eqref{est:a2est} yields
\begin{align*}
\big\|a_2(\sqrt{\mathcal L_a}) b_k\big\|_{L^2(\R^d)}
&\lesssim\bigl(\tfrac{r_k}R)^{2\mu} R^{-d(\frac12-\frac1{q'})}\|b_k\|_{L^q} \int_0^\infty
\bigl(\tfrac t{R^2}\bigr)^{\frac{j-2\mu}2-\frac d4} \bigl(1+\tfrac{t}{R^2}\bigr)^{\frac d{2q'}} e^{-\frac t{R^2}}\tfrac{dt}{t}\\
&\lesssim\bigl(\tfrac{r_k}R)^{2\mu} R^{-d(\frac12-\frac1{q'})}\|b_k\|_{L^q} ,
\end{align*}
for any $R\geq r_k$. (Note that we also rely on the fact that $\frac{j-2\mu}2-\frac d4>0$.)  Excepting the fact that it remains to justify
\eqref{I}, this completes the proof of \eqref{equ:guije} and with it, the proof of Theorem~\ref{thm:multthm}.

In view of the heat kernel bounds described in Theorem~\ref{T:heat}, \eqref{I} reduces to
\begin{align}\label{IiIiI}
\Big\|\int_{\R^d}\big(1\vee\tfrac{\sqrt{t}}{|x|}\big)^\sigma\big(1\vee\tfrac{\sqrt{t}}{|y|}\big)^\sigma e^{-\frac{|x-y|^2}{ct}}b_k(y)dy\Big\|_{L^2}
\lesssim t^{\frac{d}4} \big(t+r_k^2\bigr)^\frac{d}{2q'} \|b_k\|_{L^q}.
\end{align}
We consider four cases:

\noindent{\bf Case 1: $|x|<\sqrt{t},~|y|<\sqrt{t}$.}   Using H\"older's inequality and recalling that $q>r_0$, we estimate the contribution of this region to LHS\eqref{IiIiI} by
\begin{align*}
t^\sigma\Big\||x|^{-\sigma}\!\!\!\int_{|y|<\sqrt{t}}|y|^{-\sigma}b_k(y)dy\Big\|_{L^2(|x|<\sqrt{t})}
&\lesssim t^\sigma\big\||x|^{-\sigma}\big\|_{L^2(|x|<\sqrt{t})}\big\||y|^{-\sigma}\big\|_{L^{q'}(|y|<\sqrt{t})}\|b_k\|_{L^q}\\
&\lesssim t^{\frac{d}2(\frac12+\frac1{q'})}\|b_k\|_{L^q}.
\end{align*}

\noindent{\bf Case 2: $|x|\geq\sqrt{t},~|y|<\sqrt{t}$.} Using the Minkowski inequality and arguing as in Case 1, we estimate the contribution of this region to LHS\eqref{IiIiI} by
\begin{align*}
t^\frac\sigma2\!\!\int_{|y|<\sqrt{t}}|y|^{-\sigma}|b_k(y)|\big\|e^{-\frac{|x-y|^2}{ct}}\big\|_{L_x^2}dy
&\lesssim t^{\frac{\sigma}2+\frac{d}4}\big\||y|^{-\sigma}\big\|_{L^{q'}(|y|<\sqrt{t})}\|b_k\|_{L^q} \\
&\lesssim t^{\frac{d}2(\frac12+\frac1{q'})}\|b_k\|_{L^q}.
\end{align*}

\noindent{\bf Case 3: $|x|<\sqrt{t},~|y|\geq\sqrt{t}$.} We estimate the contribution of this region to LHS\eqref{IiIiI}
by\begin{align*}
t^\frac\sigma2\big\||x|^{-\sigma}\big\|_{L^2(|x|<\sqrt{t})}\big\|e^{-\frac{|y|^2}{ct}}\big\|_{L^{q'}}\|b_k\|_{L^q}
\lesssim t^{\frac{d}2(\frac12+\frac1{q'})}\|b_k\|_{L^q}.
\end{align*}

\noindent{\bf Case 4: $|x|\geq\sqrt{t},~|y|\geq\sqrt{t}$.} Using the Minkowski and H\"older inequalities, we estimate the contribution of this region to LHS\eqref{IiIiI} by
\begin{align*}
\Big\|\int_{|y|\geq\sqrt{t}}b_k(y)e^{-\frac{|x-y|^2}{ct}}dy\Big\|_{L^2_x}
\lesssim \int_{\R^d}|b_k(y)|\big\|e^{-\frac{|x|^2}{ct}}\big\|_{L_x^2}dy
\lesssim t^\frac{d}4\|b_k\|_{L^1}
\lesssim t^\frac{d}4r_k^\frac{d}{q'}\|b_k\|_{L^q}.
\end{align*}

This completes the proof of \eqref{IiIiI}, and with it, the proof of Theorem~\ref{thm:multthm}.
\end{proof}

We end this section with an example which shows that one cannot improve upon the range of exponents $p$ in Theorem~\ref{thm:multthm}, even if one assumes \eqref{est:multthm} holds for all $j\geq 0$.  By self-adjointness, it suffices to prove this failure  for $p\geq \frac d\sigma$ when $\sigma>0$.

To this end, let $m(\lambda):=e^{-\lambda^2}$; note that $m$ satisfies \eqref{est:multthm} for all $j\geq 0$.  Let $\varphi$ be a smooth radial cutoff supported in $B(0,1)$ such that $\varphi(x)=1$ for $|x|\leq 1/2$.  Using the heat kernel estimate provided by Theorem~\ref{T:heat}, for $|x|\leq 1$ we have
\begin{align*}
[m(\sqrt \mathcal L_a)\varphi](x)= [e^{-\mathcal L_a}\varphi](x)\gtrsim |x|^{-\sigma} \int_{|y|\leq 1/2} |y|^{-\sigma} e^{-c|x-y|^2}\, dy\gtrsim |x|^{-\sigma}.
\end{align*}
Therefore, $m(\sqrt \mathcal L_a)\varphi \notin L^p$ for any $p\geq \frac d\sigma$.

%%%%%%%%%%%%%%%%%%%%%%%%%%%%%%%%%%%%%%%%%%%%%%%%%%%%%%%%%%%%%%%%%%%%%%%%%%%%%%%%%%%%%%%%%%%%%%%%%%%%%%%%%%%%%%%%%%%%%%%%%%%%%%%%%%%%%%
\section{Littlewood--Paley theory}\label{SS:LP}
%%%%%%%%%%%%%%%%%%%%%%%%%%%%%%%%%%%%%%%%%%%%%%%%%%%%%%%%%%%%%%%%%%%%%%%%%%%%%%%%%%%%%%%%%%%%%%%%%%%%%%%%%%%%%%%%%%%%%%%%%%%%%%%%%%%%%%

In this section, we develop basic Littlewood--Paley theory, such as Bernstein and square function inequalities, adapted to the operator $\mathcal L_a$.  Our results in this section rely on Theorem~\ref{thm:M}.

Let $\phi:[0,\infty)\to[0,1]$ be a smooth function such that
\begin{align*}
\phi(\lambda)=1 \qtq{for} 0\le\lambda\le 1 \qtq{and} \phi(\lambda)=0\qtq{for} \lambda\ge 2.
\end{align*}
For each dyadic number $N\in 2^\Z$, we define
\begin{align*}
\phi_N(\lambda):=\phi(\lambda/N) \qtq{and} \psi_N(\lambda):=\phi_N(\lambda)-\phi_{N/2}(\lambda).
\end{align*}
Clearly, $\{\psi_N(\lambda)\}_{N\in \tz} $ forms a partition of unity for $\lambda\in(0,\infty)$.  We define the Littlewood--Paley projections as follows:
\begin{align*}
\po_{\le N} :=\phi_N\bigl(\sqrt{\mathcal L_a}\,\bigr), \quad \po_N :=\psi_N(\sqrt{\mathcal L_a}\,\bigr), \qtq{and} \po_{>N} :=I-\po_{\le N}.
\end{align*}

We also define another family of Littlewood--Paley projections via the heat kernel, as follows:
\begin{align*}
\tpo_{\le N} := e^{-\mathcal L_a/N^2}, \quad \tpo_N:=e^{-\mathcal L_a/N^2}-e^{-4\mathcal L_a/N^2}, \qtq{and} \tpo_{>N}:=I-\tpo_{\le N}.
\end{align*}

In what follows, we will write $P_N$, $\tilde P_N$, and so forth, to represent the analogous operators associated to the Euclidean Laplacian $-\Delta$.

\begin{lemma}[Bernstein estimates]\label{L:Bernie}
Let $1<p\leq q\le \infty$ when $a\geq0$ and let $r_0<p\leq q<r_0'=\frac d\sigma$ when $-(\frac{d-2}2)^2\leq a<0$. Then

\noindent{\rm(1)} The operators $\po_{\le N}$, $\po_{N}$, $\tpo_{N}$ and $\tpo_{\le N}$ are bounded on $L^p$;

\noindent{\rm(2)}  $\po_{\le N}$, $\po_{N}$, $\tpo_{N}$ and $\tpo_{\le N}$ are bounded from $L^p$ to $L^q$ with norm $O(N^{\frac dp-\frac dq})$;

\noindent{\rm(3)}  $N^s\|\po_N f\|_{\lpo} \sim \bigl\|(\mathcal L_a)^{\frac s2}\po_N f\bigr\|_{\lpo}$ for all $f\in C_c^{\infty}(\R^d)$ and all $s\in \R$.
\end{lemma}

\begin{proof} The first and third claims are easy consequences of Theorem~\ref{thm:M}.  Henceforth, we will consider the second claim.  As both $\po_N$ and $\po_{\leq N}$ can be written as products of $\tpo_{\leq N}$ with $L^p$-bounded multipliers by virtue of Theorem~\ref{thm:M}, it suffices to prove that  $\tpo_{\leq N}$ is bounded from $L^p$ to $L^q$.

When $a\geq 0$ (or equivalently, $\sigma\leq0$), the heat kernel obeys Gaussian bounds (cf. Theorem~\ref{T:heat}); thus by Young's inequality,
$$
\| e^{-\mathcal L_a / N^2} f\|_{L^q(\R^d)} \lesssim N^d \bigl\|e^{-cN^2|x|^2}\bigr\|_{L^r(\R^d)} \|f\|_{L^p(\R^d)}\lesssim N^{\frac dp -\frac dq}\|f\|_{L^p(\R^d)},
$$
where $r$ is determined by $1+\frac1q=\frac1r+\frac1p$.

We now turn to the case $-(\frac{d-2}2)^2\leq a<0$ (or equivalently, $\sigma>0$).  By Theorem~\ref{T:heat},
\begin{equation}\label{II}
\Big\|e^{-\mathcal L_a / N^2} f\Big\|_{L^q(\R^d)}
\lesssim N^d\Big\|\big(1\vee\tfrac{1}{N|x|}\big)^\sigma\!\!\int_{\R^d}e^{-cN^2|x-y|^2}\big(1\vee\tfrac{1}{N|y|}\big)^\sigma\big|f(y)\big|dy\Big\|
_{L^q(\R^d)}.
\end{equation}
To estimate RHS\eqref{II}, we divide into four cases.

\noindent{\bf Case 1: $|x|\leq N^{-1},~|y|\leq N^{-1}$.}  Using H\"older and recalling that $r_0<p\le q<r_0'=\frac{d}{\sigma}$, we estimate the contribution of this region to RHS\eqref{II} by
\begin{align*}
N^{d-2\sigma}&\Big\||x|^{-\sigma}\int_{|y|\leq N^{-1}}|y|^{-\sigma}\big|f(y)\big|dy\Big\| _{L^q(|x|\leq N^{-1})}\\
&\lesssim N^{d-2\sigma}\big\||x|^{-\sigma}\big\| _{L^q(|x|\leq N^{-1})}\big\||y|^{-\sigma}\big\| _{L^{p'}(|y|\leq N^{-1})}\|f\|_{L^p}\lesssim N^{\frac dp-\frac dq}\|f\|_{L^p}.
\end{align*}

\noindent{\bf Case 2: $|x|\leq N^{-1},~|y|> N^{-1}$.} Using H\"older's inequality, we estimate the contribution of this region to RHS\eqref{II} by
\begin{align*}
N^{d-\sigma}&\Big\||x|^{-\sigma}\int_{\R^d}e^{-cN^2|x-y|^2}\big|f(y)\big|dy\Big\|_{L^q(|x|\leq N^{-1})}\\
&\lesssim N^{d-\sigma}\big\||x|^{-\sigma}\big\| _{L^q(|x|\leq N^{-1})}\big\|e^{-cN^2|y|^2}\big\| _{L^{p'}}\|f\|_{L^p}\lesssim N^{\frac dp-\frac dq}\|f\|_{L^p}.
\end{align*}

\noindent{\bf Case 3: $|x|> N^{-1},~|y|\leq N^{-1}$.} Using the Minkowski and H\"older inequalities, we estimate the contribution of this region to RHS\eqref{II} by
\begin{align*}
N^{d-\sigma}&\Big\|\int_{|y|\leq N^{-1}}|y|^{-\sigma}e^{-cN^2|x-y|^2}\big|f(y)\big|dy\Big\|_{L^q}\\
&\lesssim N^{d-\sigma}\big\|e^{-cN^2|x|^2}\big\|_{L^q}\big\||y|^{-\sigma}\big\| _{L^{p'}(|y|\leq N^{-1})}\|f\|_{L^p}
\lesssim N^{\frac dp-\frac dq}\|f\|_{L^p}.
\end{align*}

\noindent{\bf Case 4: $|x|> N^{-1},~|y|> N^{-1}$.} Using Young's inequality, we estimate the contribution of this region to RHS\eqref{II} by
\begin{align*}
N^{d}\Big\|\int_{\R^d}e^{-cN^2|x-y|^2}\big|f(y)\big|dy\Big\|_{L^q}\lesssim N^{d}\big\|e^{-cN^2|x|^2}\big\|_{L^r}\|f\|_{L^p}\lesssim N^{\frac dp-\frac dq}\|f\|_{L^p},
\end{align*}
where $1+\tfrac1q=\tfrac1r+\tfrac1p.$

This completes the proof of the lemma.
\end{proof}

\begin{lemma}[Expansion of the identity]\label{L:ident}  For any $1<p<\infty$ when $a\geq0$ and for any $r_0<p<r_0'=\frac d\sigma$ when $-(\frac{d-2}2)^2\leq a<0$, we have
\begin{equation*}
f(x) = \sum_{N\in 2^\Z} \bigl[\po_N f\bigr](x) = \sum_{N\in 2^\Z} \bigl[\tpo_N f\bigr](x),
\end{equation*}
as elements of $L^p(\R^d)$.  In particular, the sums converge in $L^p(\R^d)$.
\end{lemma}

\begin{proof}
For $p=2$ this follows from the spectral theorem and the fact that zero is not an eigenvalue of $\mathcal L_a$.  On the other hand, Theorem~\ref{thm:M} guarantees that partial sums are $L^p$ bounded for all $p$ allowed by the lemma; this allows one to upgrade convergence in $L^2$ to convergence in $L^p$ by a simple density/interpolation argument.
\end{proof}

As a direct consequence of Theorem~\ref{thm:M}, we obtain two-sided square functions estimates; see \cite[\S IV.5]{Stein} for the requisite argument.

\begin{theorem}[Square function estimates]\label{T:sq}
Fix $s\geq 0$.  Let $1<p<\infty$ when $a\geq0$ and let $r_0<p<r_0'=\frac d\sigma$ when $-(\frac{d-2}2)^2\leq a<0$. Then for any $f\in C_c^{\infty}(\R^d)$,
\begin{align*}
\biggl\|\biggl(\sum_{N\in2^\Z} N^{2s}| \po_N f|^2\biggr)^{\!\!\frac12}\biggr\|_{\lpo} \!\!    \sim \bigl\|\mathcal L_a^{\frac s2} f \bigr\|_{\lpo}
        \sim \biggl\|\biggl(\sum_{N\in2^\Z} N^{2s}| (\tpo_N)^k f|^2\biggr)^{\!\!\frac 12}\biggr\|_{\lpo},
\end{align*}
provided the integer $k\geq 1$ satisfies $2k > s$.
\end{theorem}

Note that the function $\lambda \mapsto e^{-\lambda^2/N^2}-e^{-4\lambda^2/N^2}$ used to define $\tpo_N$ only vanishes to second order at $\lambda=0$. The restriction $k>s/2$ ensures that $N^s (\mathcal L_a)^{-s} (\tpo_N)^k$ is actually a Mikhlin multiplier.  In the next section we will take $k=1$, because Theorem~\ref{thm:equivsobolev} is only an assertion about $0<s<2$.

%%%%%%%%%%%%%%%%%%%%%%%%%%%%%%%%%%%%%%%%%%%%%%%%%%%%%%%%%%%%%%%%%%%%%%%%%%%%%%%%%%%%%%%%%%%%%%%%%%%%%%%%%%%%%%%%%%%%%%%%%%%%%%%%%%%%%%%%%%%%%%

%%%%%%%%%%%%%%%%%%%%%%%%%%%%%%%%%%%                                %%%%%%%%%%%%%%%%%%%%%%%%%%%%%%%%%%%%%%%%%%%%%%%%%%%%%%%%%%%%%%%%%%%%%%

%%%%%%%%%%%%%%%%%%%%%%%%%%%%%%%%%%%%%%%%%%%%%%%%%%%%%%%%%%%%%%%%%%%%%%%%%%%%%%%%%%%%%%%%%%%%%%%%%%%%%%%%%%%%%%%%%%%%%%%%%%%%%%%%%%%%%%%%%%%%%%

\section{Generalized Riesz transforms}
In this section we prove Theorem~\ref{thm:equivsobolev}. We start by estimating the difference between the kernels of the Littlewood--Paley projection operators adapted to the Euclidean Laplacian and $\mathcal L_a$, respectively.

\begin{lemma}[Difference of kernels]\label{lem:tedke}
Let $K_N(x,y):=\big(\tilde P_N-\tilde P_N^a\big)(x,y)$.

\noindent $(1)$ For $a\geq0$, there exists $c>0$ such that
\begin{equation}\label{est:6.1est1}
\big|K_N(x,y)\big|\lesssim
N^d\max\big\{1,N(|x|+|y|)\big\}^{-2}e^{-cN^2|x-y|^2},
\end{equation}
uniformly for $x,y\in\R^d\setminus \{0\}.$

\noindent $(2)$  For $-(\frac{d-2}2)^2\leq a<0$, there exists $c>0$ such that for any $M>0$,
\begin{equation}\label{est:6.1est1al0}
\big|K_N(x,y)\big|\lesssim \begin{cases}
N^{d-2\sigma}\left(|x||y|\right)^{-\sigma}, \qquad\qquad\quad &|x|,|y|\leq N^{-1},\\
N^d(N|x|)^{-\sigma} (N|y|)^{-M}, &2|x|\leq N^{-1}\leq|y|,\\
N^d(N|y|)^{-\sigma} (N|x|)^{-M}, &2|y|\leq  N^{-1}\leq|x|,\\
N^{d-2}(|x|+|y|)^{-2}e^{-cN^2|x-y|^2},  &|x|,|y|\geq \frac12N^{-1},
\end{cases}
\end{equation}
uniformly for $x,y\in\R^d\setminus \{0\}.$
\end{lemma}

\begin{proof}
By the definition of the Littlewood--Paley projections, we may write
\begin{align}\label{k0}
K_N(x,y)&=\bigl[e^{\Delta/N^2}-e^{4\Delta/N^2}\bigr](x,y)-\bigl[e^{-\mathcal L_a/N^2}-e^{-4\mathcal L_a/N^2}\bigr](x,y)\\\nonumber
&=\bigl[e^{\Delta/N^2}-e^{-\mathcal L_a/N^2}\bigr](x,y)-\bigl[e^{4\Delta/N^2}-e^{-4\mathcal L_a/N^2}\bigr](x,y).
\end{align}

Let us begin with the case $a\geq 0$.  By the maximum principle, $0\leq e^{-t\mathcal L_a}(x,y) \leq e^{t\Delta}(x,y)$ and consequently,
$$
|K_N(x,y)| \leq e^{\Delta/N^2}(x,y) + e^{4\Delta/N^2}(x,y) \lesssim N^d e^{-N^2|x-y|^2/16}.
$$
This suffices to verify \eqref{est:6.1est1} when $|x|+|y|\leq N^{-1}$ or $|x|\geq 2|y|$ or $|y|\geq2|x|$.  Thus, we need only prove \eqref{est:6.1est1} when $|x|+|y|\geq N^{-1}$ with $|x|\sim |y|$.

By Duhamel's formula,
\begin{equation}\label{heat diff duhamel}
\bigl[e^{t\Delta} - e^{-t\mathcal L_a}\bigr](x,y)=a\int_0^t \int_{\R^d} e^{(t-s)\Delta}(x,z) \tfrac1{|z|^2} e^{-s\mathcal L_a}(z,y)\,dz\,ds.
\end{equation}
Thus, by exploiting the maximum principle again and using the fact that
$$
\tfrac{|x-z|^2}{t-s} + \tfrac{|z-y|^2}{s} \geq  \tfrac{|x-z|^2+|z-y|^2}{t} = \tfrac{2|z-\frac{x+y}{2}|^2}{t} + \tfrac{|x-y|^2}{2t} \geq \tfrac{|x-y|^2}{2t}
$$
we can deduce that
\begin{align*}
\big|&e^{t\Delta}(x,y) - e^{-t\mathcal L_a}(x,y)\big|\\
&\lesssim \int_0^t\int_{\R^d}(t-s)^{-\frac{d}2}e^{-\frac{|x-z|^2}{4(t-s)}}\tfrac1{|z|^2} s^{-\frac{d}2}e^{-\frac{|z-y|^2}{4s}}dz\,ds\\
&\lesssim t^{-\frac{d}2}e^{-\frac{|x-y|^2}{16t}}\Big[\int_0^\frac{t}2\int_{\R^d}\frac{s^{-\frac{d}2}}{|z+y|^2}e^{-\frac{|z|^2}{8s}}dz\,ds+
    \int_{\frac{t}2}^t\int_{\R^d}\frac{(t-s)^{-\frac{d}2}}{|z+x|^2}e^{-\frac{|z|^2}{8(t-s)}}dz\,ds\Big]\\
&\lesssim t^{-\frac{d}2}e^{-\frac{|x-y|^2}{16t}}\Big[\int_0^\frac12\int_{\R^d}\frac{s^{-1}}{|z+\frac{y}{\sqrt{ts}}|^2}e^{-\frac{|z|^2}{8}}dzds+
\int_0^\frac12\int_{\R^d}\frac{s^{-1}}{|z+\frac{x}{\sqrt{ts}}|^2}e^{-\frac{|z|^2}{8}}dzds\Big].
\end{align*}
Using the fact that for $d\geq3$
$$
\int_{\R^d}e^{-\frac{|z|^2}{8}}\tfrac1{|z-x_0|^2}dz\lesssim |x_0|^{-2}\qtq{for} |x_0|\geq1,
$$
we obtain
\begin{align*}
\big|e^{t\Delta}(x,y)-e^{-t\mathcal L_a}(x,y)\big| \lesssim t^{-\frac{d}2}e^{-\frac{|x-y|^2}{16t}}\frac{t}{(|x|+|y|)^2}\qtq{for all} |x|,|y|\geq\sqrt{t}.
\end{align*}
Taking $t=N^{-2}$ and $t=4N^{-2}$, the remaining case of \eqref{est:6.1est1} now follows.

Next we consider the case when $-(\frac{d-2}2)^2\leq a<0$.  By the maximum principle and Theorem~\ref{T:heat}, we have
\begin{equation*}
|K_N(x,y)| \leq e^{-\mathcal L_a/N^2}(x,y)+e^{-4\mathcal L_a/N^2}(x,y)\lesssim N^{d}e^{-{cN^2|x-y|^2}}\big(1\vee\tfrac{1}{N|x|}\big)^\sigma\big(1\vee\tfrac{1}{N|y|})^\sigma.
\end{equation*}
This directly justifies \eqref{est:6.1est1al0}, except when $|x|,|y|\geq \frac12 N^{-1}$ and $|x|\sim|y|$.  Let us now focus on this remaining scenario.  In view of Theorem~\ref{T:heat}, the analysis of this case
via \eqref{heat diff duhamel} follows as in the case $a\geq 0$, except in the regime where $z$ is small.  More precisely, to finish the proof,  it suffices to show
\begin{equation}\label{E:last dh}
 \int_0^t \int_{|z|^2 < \frac{s}{4}} e^{(t-s)\Delta}(x,z) \tfrac1{|z|^2} e^{-s\mathcal L_a}(z,y)\,dz\,ds \lesssim \frac{e^{-c|x-y|^2/t}}{t^{(d-2)/2} (|x|^2+|y|^2)}
\end{equation}
when $|x|,|y|\geq \sqrt{t}$ and $|x|\sim|y|$.

Under these conditions on $x$, $y$, and $z$, we have $|x-z|^2\sim|y-z|^2\sim(|x|^2+|y|^2)$.  Thus invoking Theorem~\ref{T:heat} and using $0\leq s\leq t$, we deduce that
\begin{align*}
\text{LHS\eqref{E:last dh}} \lesssim \int_0^t \int_{|z|^2 < t} \tfrac{t^{\sigma/2}}{|z|^{2+\sigma}} \,dz\, [s(t-s)]^{-d/2} \exp\bigl\{-\tfrac{ct}{s(t-s)}(|x|^2+|y|^2)\bigr\}\,ds.
\end{align*}
To continue, we employ the elementary inequality
$$
\sup_{0<s<t} [s(t-s)]^{-d/2} \exp\bigl\{-\tfrac{ct(|x|^2+|y|^2)}{s(t-s)}\bigr\} \lesssim [ t (|x|^2+|y|^2) ]^{-d/2} \exp\bigl\{-\tfrac{c}{t}(|x|^2+|y|^2)\bigr\},
$$
as well as the observation that $\sigma+2<d$.  This then allows us to conclude that
\begin{align*}
\text{LHS\eqref{E:last dh}} \lesssim \bigl[(|x|^2+|y|^2)\bigr]^{-d/2} \exp\bigl\{-\tfrac{c}{t}(|x|^2+|y|^2)\bigr\}\,ds,
\end{align*}
which easily implies \eqref{E:last dh} since $|x|,|y|\geq \sqrt{t}$.
\end{proof}

Using Lemma~\ref{lem:tedke}, we now estimate the difference between the two Littlewood--Paley square functions adapted to the Euclidean Laplacian and $\mathcal L_a$, respectively.

\begin{proposition}[Difference of square functions]\label{prop:lpdiffer}
Let  $d\geq 3$ and $0<s<2$. Assume $1<p<\infty$ when $a\geq0$ and $\max\{1,\frac{d}{d+s-\sigma}\}<p<\frac{d}{\sigma}$ when $-(\frac{d-2}2)^2\leq a<0$. Then for any $f\in C_c^{\infty}(\R^d)$, we have
\begin{align}\label{sfd}
\biggl\|\biggl(\sum_{N\in 2^{\Z}}N^{2s}\bigl|\tilde P_Nf\bigr|^2\biggr)^{\frac 12} -\biggl(\sum_{N\in 2^{\Z}}N^{2s}\bigl|\tilde P_N^a f\bigr|^2\biggr)^{\frac 12}\biggr\|_{L^p(\R^d)}
\lesssim \biggl\|\frac{f}{|x|^s}\biggr\|_{L^p(\R^d)}.
\end{align}
\end{proposition}

\begin{proof}
By the triangle inequality and the embedding $\ell^1\hookrightarrow \ell^2$, we have
\begin{align*}
\text{LHS\eqref{sfd}}
&\lesssim \biggl\|\biggl(\sum_{N\in 2^{\Z}}N^{2s}\bigl|\bigl[\tilde P_N-\tilde P_N^a\bigr]f\bigr|^2\biggr)^{\frac12}\biggr\|_{L^p(\R^d)}\\
&\lesssim \biggl\|\biggl(\sum_{N\in 2^{\Z}}N^{2s}\biggl|\int K_N(x,y)f(y)\,dy\biggr|^2\biggr)^{\frac 12}\biggr\|_{L^p(\R^d)}\\
&\lesssim \biggl\|\int\ \biggl(\sum_{N\in 2^{\Z}} N^s |K_N(x,y)|\biggr)\ |f(y)|\,dy\biggr\|_{L^p(\R^d)}.
\end{align*}
where $K_N$ is as in Lemma~\ref{lem:tedke}.

\noindent {\bf Case 1: $a\geq0$.}  We decompose the sum into the contribution of high frequencies $N\geq (|x|+|y|)^{-1}$ and that of low frequencies $N\leq (|x|+|y|)^{-1}$.  We first discuss the contribution of high frequencies.

For $N\geq (|x|+|y|)^{-1}$, using Lemma~\ref{lem:tedke} we find that for any $M,L>0$,
\begin{equation*}
|K_N(x,y)|\lesssim N^d\begin{cases}
 \big[N(|x|+|y|)\big]^{-M},\qquad\qquad &|x|\geq  2|y| \text{ or } |y|\geq2|x|,\\
 \big[N(|x|+|y|)\big]^{-2}\big[N|x-y|\big]^{-L}, &|x|\sim|y|.
\end{cases}
\end{equation*}
Choosing $M>d+s$ and $d+s-2<L<d$ (which is possible for $s<2$), we estimate the contribution of high frequencies as follows:
\begin{align*}
&\biggl\|\int\sum_{N\geq(|x|+|y|)^{-1}} N^s |K_N(x,y)| |f(y)|\,dy\biggr\|_{L^p}\\
&\leq\biggl\|\int_{2|y|\leq|x|}\sum_{N\geq(|x|+|y|)^{-1}}\frac{N^{d+s}}{\big[N(|x|+|y|)\big]^M}|f(y)|\,dy\biggr\|_{L^p}\\
&\quad+\biggl\|\int_{2|x|\leq|y|}\sum_{N\geq(|x|+|y|)^{-1}}\frac{N^{d+s}}{\big[N(|x|+|y|)\big]^M} |f(y)|\,dy\biggr\|_{L^p}\\
&\quad +\biggl\|\int_{|x|\sim |y|}\sum_{N\geq(|x|+|y|)^{-1}}\frac{N^{d+s}}{\big[N(|x|+|y|)\big]^2\big[N|x-y|\big]^L}|f(y)|\,dy\biggr\|_{L^p}\\
&\lesssim\biggl\|\int_{2|y|\leq|x|}\frac{|y|^{s}}{(|x|+|y|)^{d+s}}\frac{|f(y)|}{|y|^s}\,dy\biggr\|_{L^p}+\biggl\|\int_{2|x|\leq|y|}\frac{|y|^{s}}{(|x|+|y|)^{d+s}}\frac{|f(y)|}{|y|^s}\,dy\biggr\|_{L^p}\\
&\quad+\biggl\|\int_{|x|\sim |y|}\frac{|y|^{s}}{(|x|+|y|)^{d+s-L}|x-y|^L}\frac{|f(y)|}{|y|^s}\,dy\biggr\|_{L^p}\\
&\triangleq I_1+I_2+I_3.
\end{align*}

Using Lemma~\ref{lem:schur}, we will estimate these three integrals below in terms of RHS\eqref{sfd}.  It turns out that these three integrals also control the contribution from the low frequencies.  Indeed, for $N\leq (|x|+|y|)^{-1}$, Lemma~\ref{lem:tedke} guarantees that $|K_N(x,y)|\lesssim N^d$ and consequently,
\begin{align*}
\biggl\|\int\biggl(\sum_{N\leq(|x|+|y|)^{-1}} \!\!N^s|K_N(x,y)|\biggr)|f(y)|\,dy\biggr\|_{L^p}
&\lesssim \biggl\|\int \frac{ |y|^s}{(|x|+|y|)^{d+s}}\frac{|f(y)|}{|y|^s}\,dy\biggr\|_{L^p}\\
&\lesssim I_1+I_2+I_3.
\end{align*}

We first consider the contribution of $I_1$.  On the corresponding region of integration, the kernel becomes
$$
K(x,y)=\frac{|y|^{s}}{(|x|+|y|)^{d+s}}\sim|y|^{s}|x|^{-(d+s)}.
$$
As $s>0$, a simple computation shows that
$$
\sup_{x\in \R^d}\int_{2|y|\leq|x|}K(x,y)\, dy+\sup_{y\in \R^d}\int_{2|y|\leq |x|} K(x,y)\,dx\lesssim1.
$$
Thus, Schur's test shows the contribution of $I_1$ is acceptable.

We now consider the contribution of $I_2$.  On the corresponding region of integration, the kernel takes the form
$$
K(x,y)=\frac{|y|^{s}}{(|x|+|y|)^{d+s}}\sim|y|^{-d}.
$$
We will apply Lemma~\ref{lem:schur} with weight given by
$$
w(x,y):=\big(\tfrac{|x|}{|y|}\big)^\alpha \qtq{for} 0<\alpha<p'd.
$$
That hypothesis \eqref{est:scha1} is satisfied in this setting follows from
$$
\int_{2|x|\leq|y|}w(x,y)^\frac1p\big|K(x,y)\big|dy\lesssim|x|^{\frac{\alpha}p}\int_{2|x|\leq|y|}|y|^{-d-\frac{\alpha}p}dy\lesssim1,
$$
while hypothesis \eqref{est:scha2} follows from
\begin{align*}
\int_{2|x|\leq|y|}w(x,y)^{-\frac1{p'}}\big|K(x,y)\big|dx\lesssim&|y|^{-d+\frac{\alpha}{p'}}\int_{2|x|\leq|y|}|x|^{-\frac{\alpha}{p'}}dx \lesssim1.
\end{align*}
Thus, Lemma~\ref{lem:schur} shows that the contribution of $I_2$ is acceptable.

Finally, we consider the contribution of $I_3$.  On the corresponding region of integration, the kernel becomes
$$
K(x,y)=\frac{|y|^{s}}{(|x|+|y|)^{d+s-L}|x-y|^L}\sim|x|^{-d+L}|x-y|^{-L}\sim|y|^{-d+L}|x-y|^{-L}.
$$
Recalling that $L<d$, we estimate
\begin{align*}
&\sup_{x\in \R^d}\int_{|x|\sim|y|}K(x,y)\, dy+\sup_{y\in \R^d}\int_{|x|\sim|y|} K(x,y)\,dx\\
&\lesssim\sup_{x\in \R^d} |x|^{-d+L} \int_{|x-y|\lesssim |x|} |x-y|^{-L}\, dy+\sup_{y\in \R^d} |y|^{-d+L} \int_{|x-y|\lesssim |y|} |x-y|^{-L}\, dx \lesssim1.
\end{align*}
Thus, Schur's test ensures that the contribution of $I_3$ is acceptable.

This completes the proof of the proposition in the case $a\geq 0$.

\medskip

\noindent{\bf Case 2: $-(\frac{d-2}2)^2\leq a<0$.}  As before, we decompose the sum into the contribution of high frequencies $N\geq (|x|\vee|y|)^{-1}$ and that of low frequencies
$N\leq (|x|\vee|y|)^{-1}$.

We consider first the low frequencies.  Using Lemma \ref{lem:tedke} and the fact that $d+s\geq 2\sigma$, we estimate
\begin{align*}
&\biggl\|\int_{\R^d}\sum_{N\leq(|x|\vee|y|)^{-1}} N^s |K_N(x,y)| |f(y)|\,dy\biggr\|_{L^p}\\
&\lesssim\biggl\|\int_{\R^d}\!\!\sum_{N\leq(|x|\vee|y|)^{-1}}\!\!\frac{N^{d+s-2\sigma}}{|x|^\sigma|y|^\sigma}|f(y)|\,dy\biggr\|_{L^p}
\lesssim\biggl\|\int_{\R^d}\frac{|y|^{s-\sigma}}{(|x|+|y|)^{d+s-2\sigma}|x|^\sigma}\frac{|f(y)|}{|y|^s}\,dy\biggr\|_{L^p}\\
&\lesssim \biggl\|\int_{2|y|\leq |x|}\frac{|y|^{s-\sigma}}{|x|^{d+s-\sigma}}\frac{|f(y)|}{|y|^s}\,dy\biggr\|_{L^p}
    +\biggl\|\int_{2|x|\leq |y|}\frac{1}{|y|^{d-\sigma}|x|^{\sigma}}\frac{|f(y)|}{|y|^s}\,dy\biggr\|_{L^p} \\
&\quad + \biggl\|\int_{|x|\sim |y|}\frac{1}{|x|^d}\frac{|f(y)|}{|y|^s}\,dy\biggr\|_{L^p} =: I\!I_1 + I\!I_2 + I\!I_3.
\end{align*}

Before turning our attention to the estimation of these three integrals, let us pause to consider the high frequency contribution. To this end, with $x$ fixed, we divide the integral in the $y$ variable into three regions: $\{2|y|\leq |x|\}$, $\{2|x|\leq |y|\}$, and $\{ |x|/2\leq |y|\leq 2|x|\}$.

If $2|y|\leq |x|$, then choosing $M>d+s-\sigma$, Lemma~\ref{lem:tedke} yields
\begin{align*}
\sum_{N\geq(|x|\vee|y|)^{-1}} N^s |K_N(x,y)| &\lesssim \sum_{(2|y|)^{-1}\geq N\geq(|x|\vee|y|)^{-1}}N^{d+s}(N|y|)^{-\sigma} (N|x|)^{-M} \\
&\qquad + \sum_{N\geq(2|y|)^{-1}\geq(|x|\vee|y|)^{-1}}N^{d+s}(N|x|)^{-M-\sigma} \\
&\lesssim |y|^{-\sigma} |x|^{\sigma-s-d}.
\end{align*}
The contribution of this region is therefore bounded by $I\!I_1$.

Arguing similarly, we see that when $2|x|\leq |y|$ and $M>d+s-\sigma$,
\begin{align*}
\sum_{N\geq(|x|\vee|y|)^{-1}} N^s |K_N(x,y)| &\lesssim \sum_{(2|x|)^{-1}\geq N\geq(|x|\vee|y|)^{-1}} N^{d+s}(N|x|)^{-\sigma}(N|y|)^{-M} \\
&\qquad + \sum_{N\geq(2|x|)^{-1}\geq(|x|\vee|y|)^{-1}}N^{d+s}(N|y|)^{-M-\sigma},  \\
&\lesssim |x|^{-\sigma} |y|^{\sigma-s-d}.
\end{align*}
whose contribution is controlled by $I\!I_2$.

The third and final high-frequency contribution comes from the region where $|x|/2\leq |y|\leq 2|x|$.  Choosing $d+s-2<L<d$, on this region we estimate
\begin{align*}
\sum_{N\geq(|x|\vee|y|)^{-1}} N^s |K_N(x,y)|
    &\lesssim \sum_{N\geq(|x|\vee|y|)^{-1}} \frac{N^{d+s}}{\big[N(|x|+|y|)\big]^2\big[N(|x-y|)\big]^L} \\
&\lesssim (|x|+|y|)^{L-d-s}|x-y|^{-L}.
\end{align*}
Our earlier analysis of $I_3$ then shows that the resulting contribution is acceptable.

To complete the proof of Proposition~\ref{prop:lpdiffer} it remains only to estimate the three integrals $I\!I_1$, $I\!I_2$, and $I\!I_3$.
We begin with $I\!I_1$, applying Lemma~\ref{lem:schur} with weight
$$
w(x,y)=\big(\tfrac{|x|}{|y|}\big)^\alpha \qtq{with} (\sigma-s)p'<\alpha<p(d+s-\sigma).
$$
That hypothesis \eqref{est:scha1} is satisfied in this setting follows from
$$
\int_{2|y|\leq|x|}w(x,y)^\frac1p\frac{|y|^{s-\sigma}}{|x|^{d+s-\sigma}}\,dy\lesssim|x|^{-d-s+\sigma+\frac{\alpha}p}\int_{2|y|\leq|x|}|y|^{s-\sigma-\frac{\alpha}p} \,dy\lesssim1,
$$
while hypothesis \eqref{est:scha2} follows from
$$
\int_{2|y|\leq|x|}w(x,y)^{-\frac1{p'}}\frac{|y|^{s-\sigma}}{|x|^{d+s-\sigma}}\,dx\lesssim|y|^{s-\sigma+\frac{\alpha}{p'}}\int_{2|y|\leq|x|}|x|^{-(d+s-\sigma+\frac{\alpha}{p'})} \,dx\lesssim1.
$$
Thus, Lemma~\ref{lem:schur} shows that the contribution of $I\!I_1$ is acceptable.

We now consider the contribution of $I\!I_2$.  We use Lemma~\ref{lem:schur} with weight
$$
w(x,y)=\big(\tfrac{|x|}{|y|}\big)^\alpha \qtq{with} \sigma p<\alpha<p'(d-\sigma).
$$
That hypothesis \eqref{est:scha1} is satisfied in this setting follows from
$$
\int_{2|x|\leq|y|}w(x,y)^\frac1p\frac{1}{|y|^{d-\sigma}|x|^\sigma}\,dy\lesssim|x|^{-\sigma+\frac{\alpha}p}\int_{2|x|\leq|y|}|y|^{-d-\frac{\alpha}p+\sigma} \,dy\lesssim1,
$$
while hypothesis \eqref{est:scha2} is derived from
\begin{align*}
\int_{2|x|\leq|y|}w(x,y)^{-\frac1{p'}}\frac{1}{|y|^{d-\sigma}|x|^\sigma}\,dx\lesssim&|y|^{-d+\sigma+\frac{\alpha}{p'}}\int_{2|x|\leq|y|}|x|^{-\sigma-\frac{\alpha}{p'}}\,dx \lesssim1.
\end{align*}
Thus by Lemma~\ref{lem:schur}, the contribution of this term to RHS\eqref{sfd} is acceptable.

It remains only to control the contribution of $I\!I_3$. Noting that
$$
\sup_{x\in \R^d}\int_{|x|\sim|y|} \frac{1}{|x|^d}\, dy+\sup_{x\in \R^d}\int_{|x|\sim|y|}\frac{1}{|x|^d}\, dx\lesssim1,
$$
Schur's Lemma ensures that this integral is acceptable.

This completes the proof of the proposition in the case $-(\frac{d-2}2)^2\leq a<0$.
\end{proof}

We now have all the necessary ingredients to complete the proof of Theorem~\ref{thm:equivsobolev}.

\begin{proof}[Proof of Theorem~\ref{thm:equivsobolev}] Fix $d\geq 3$, $a\geq -\big(\tfrac{d-2}2\big)^2$, and $0<s<2$. Assume $\frac{s+\sigma}d<\frac1p<\min\{1, \frac{d-\sigma}d\}$. Using the triangle inequality, Theorem~\ref{T:sq}, Proposition~\ref{prop:lpdiffer}, and Proposition~\ref{P:hardy}, for $f\in C_c^\infty(\R^d)$ we estimate
\begin{align*}
\||\nabla|^s f\|_p
&\sim \biggl\|\biggl(\sum_{N\in 2^{\Z}} N^{2s}\bigl|\tilde P_N f\bigr|^2\biggr)^{\frac 12}\biggr\|_p\\
&\lesssim \biggl\|\biggl(\sum_{N\in 2^{\Z}}N^{2s}\bigl|\tilde P_N^a f\bigr|^2\biggr)^{\frac12}\biggr\|_p
    +\biggl\|\biggl(\sum_{N\in 2^{\Z}}N^{2s}\bigl|\bigl[\tilde P_N-\tilde P_N^a\bigr] f\bigr|^2\biggr)^{\frac 12}\biggr\|_p\\
&\lesssim \|\mathcal L_a^{\frac s2} f\|_p+\biggl\|\frac{f}{|x|^s}\biggr\|_p\lesssim \|\mathcal L_a^{\frac s2} f\|_p.
\end{align*}
Arguing similarly and using Lemma~\ref{lem:classhardy} in place of Proposition~\ref{P:hardy}, for $\max\{\frac sd, \frac\sigma d\}<\frac1p<\min\{1, \frac{d-\sigma}d\}$ we estimate
\begin{align*}
\|\mathcal L_a^{\frac s2} f\|_p
&\sim \biggl\|\biggl(\sum_{N\in2^{\Z}}N^{2s}\bigl|\tilde P_N^a f\bigr|^2\biggr)^{\frac 12}\biggr\|_p\\
&\lesssim \biggl\|\biggl(\sum_{N\in 2^{\Z}}N^{2s}\bigl|\tilde P_Nf\bigr|^2\biggr)^{\frac 12}\biggr\|_p
    +\biggl\|\biggl(\sum_{N\in 2^{\Z}} N^{2s}\bigl|\bigl[\tilde P_N-\tilde P_N^a\bigr] f\bigr|^2\biggr)^{\frac 12}\biggr\|_p\\
&\lesssim \||\nabla|^s f\|_p+\biggl\|\frac{f}{|x|^s}\biggr\|_p\lesssim \||\nabla|^s f\|_p,
\end{align*}

This completes the proof of the theorem.
\end{proof}

\begin{center}

\end{center}

\end{document}